\documentclass[10pt,reqno]{amsart}
\usepackage[nodisplayskipstretch]{setspace}
\usepackage[brazilian,english]{babel} 
\usepackage[T1]{fontenc}
\usepackage{amsfonts,amsthm,amssymb,amsmath,mathrsfs}  
\usepackage{graphicx,color}       
\usepackage{latexsym}       
\usepackage{overpic}        
\usepackage[colorlinks=false, linktocpage=true]{hyperref}
\usepackage{graphicx}
\usepackage{multicol}
\usepackage{color}
\usepackage{multirow} 

\newtheorem{theorem}{Theorem}[section]
\newtheorem{corollary}[theorem]{Corollary}

\newtheorem{lemma}[theorem]{Lemma}
\newtheorem{proposition}[theorem]{Proposition}

\theoremstyle{definition}
\newtheorem{definition}[theorem]{Definition}
\newtheorem{remark}[theorem]{Remark}

\newcommand{\E}{E_{\mathbf{0}}}
\newcommand{\ka}{\kappa}
\newcommand{\R}{\mathbb{R}}

\newcommand{\C}{\mathbb{C}}
\newcommand{\nnabla}{\nabla\hspace{-0.4cm}\not}
\newcommand{\Z}{\mathbb{Z}}

\newcommand{\Q}{\mathcal{Q}}

\newcommand{\psib}{\mbox{\boldmath$\psi$}}

\newcommand{\ub}{\mathbf{u}}

\newcommand{\zb}{\mathbf{z}}
\newcommand{\Hb}{\mathbf{H}}

\newcommand{\Lb}{\mathbf{L}}
\newcommand{\Wb}{\mathbf{W}}
\newcommand{\betab}{\boldsymbol{\beta}}

\newcommand{\M}{\mathscr{M}}
\numberwithin{equation}{section}

\begin{document}

\title[Scattering for NLS systems]{Scattering for quadratic-type Schr\"{o}dinger systems  in dimension five without mass-resonance}

\author[N. Noguera]{Norman Noguera}
\address{SM-UCR, Ciudad Universitaria Carlos Monge Alfaro, Departamento de Ciencias Naturales, Apdo: 111-4250, San Ram\'on, Alajuela, Costa Rica}
\email{norman.noguera@ucr.ac.cr}

\author[A. Pastor]{Ademir Pastor}
\address{IMECC-UNICAMP, Rua S\'ergio Buarque de Holanda, 651, 13083-859, Cam\-pi\-nas-SP, Bra\-zil}
\email{apastor@ime.unicamp.br}

\begin{abstract}
In this paper we study the scattering of non-radial solutions in the energy space  to coupled system of nonlinear Schr\"{o}dinger equations with quadratic-type growth interactions in dimension five without the mass-resonance condition. Our approach is based on the recent technique introduced by Dodson and Murphy in \cite{dodsonmurphy2018}, which relies on an interaction Morawetz estimate. It is proved that any solution below the ground states scatters in time.  
\end{abstract}



\maketitle

\tableofcontents

\section{Introduction}

In this paper we continue our study concerning the long time behavior of global solutions to $l$-component Schr\"odinger systems with a quadratic-growth nonlinearity. More precisely, we consider the  following initial-value problem 
\begin{equation}\label{system1}
	\begin{cases}
		\displaystyle i\alpha_{k}\partial_{t}u_{k}+\gamma_{k}\Delta u_{k}-\beta_{k} u_{k}=-f_{k}(u_{1},\ldots,u_{l}),\\
		(u_{1}(x,0),\ldots,u_{l}(x,0))=(u_{10},\ldots,u_{l0}),\qquad k=1,\ldots l,
	\end{cases}
\end{equation}
where $u_{1},\ldots,u_{l}$ are complex-valued functions on the variables $(x,t)\in\R^n\times\R$, $\Delta$ stands for the standard Laplacian operator, $\alpha_{k}, \gamma_{k}>0$, $\beta_{k}\geq0$ are real constants and  the nonlinearities $f_{k}$ satisfy a quadratic-type growth.

The study of nonlinear Schr\"odinger systems with quadratic nonlinearities has attracted a lot of attention in recent years. One of the most studied models is the following one
\begin{equation}\label{system1J}
\begin{cases}
\displaystyle i\partial_{t}u_1+\Delta u_1=-2\overline{u}_1u_2,\\
\displaystyle i\partial_{t}u_2+\kappa\Delta u_2=- u^{2}_1,
\end{cases}
\end{equation}
which may be seen as a non-relativistic version  of some Klein-Gordon systems (see\cite{Hayashi}). It can also be derived as a model in nonlinear optics (see \cite{Colin2}). To the best of our knowledge, from the mathematical point of view, the study of the initial-value problem associated with \eqref{system1J} was initiated in \cite{Hayashi}, where several  properties  were established. To motivate our discussion, let us recall some of the results. The local  well-posedness, in the spaces $L^{2}(\R^n)$ and $H^{1}(\R^n)$, was proved for $1\leq n\leq 4$ and $1\leq n\leq 6$, respectively. Due to the conservation of mass and energy the local results may be extended to  global ones  provided $1\leq n\leq 3$. In dimension $n=4$, system \eqref{system1J} is $L^2$-critical in the sense of scaling. In consequence,  for initial data in $H^{1}(\R^{4})$, a sharp threshold for global well-posedness was proved depending on the size of the initial data when compared to the associated ground states (we refer the reader to \cite{Hayashi} for further properties).   Other results concerning global well-posedness and blow-up for \eqref{system1J} have also appeared in the current literature. Indeed, the dichotomy global existence versus blow-up in finite time in $H^1(\R^{5}$) was discussed in \cite{hamano2018global} and \cite{NoPa} (see also \cite{OgawaUriya}). In \cite{dinh2020blow} the author studied the stability of ground states (for $1\leq n\leq3$) as well as the characterization of minimal mass blow-up solutions (for $n=4$). Existence of blow-up/grow-up solutions was also studied in \cite{inui2020blow}. The strong instability of standing waves in the case $n=5$ was established in \cite{dinh2020instability}.

Mostly motivated by the results in \cite{Hayashi}, in \cite{NoPa2} we started the study of \eqref{system1} with general quadratic-type nonlinearity. More precisely,  with a slightly modification in \ref{H4}, we assumed the following (see notation below)

\newtheorem{thmx}{}
\renewcommand\thethmx{(H1)}
\begin{thmx}\label{H1}
	We have
	\begin{align*}
		f_{k}(\mathbf{0})=0, \qquad  k=1,\ldots,l.
	\end{align*}
\end{thmx}

\renewcommand\thethmx{(H2)}
\begin{thmx}\label{H2}
For any $\mathbf{z},\mathbf{z}'\in \C^{l}$ we have
	\begin{equation*}
		\begin{split}
			\left|\frac{\partial }{\partial z_{m}}[f_{k}(\mathbf{z})-f_{k}(\mathbf{z}')]\right|+ \left|\frac{\partial }{\partial \overline{z}_{m}}[f_{k}(\mathbf{z})-f_{k}(\mathbf{z}')]\right|&\lesssim \sum_{j=1}^{l}|z_{j}-z_{j}'|,\qquad k,m=1,\ldots,l,
		\end{split}
	\end{equation*}
\end{thmx}

\renewcommand\thethmx{(H3)}
\begin{thmx}\label{H3}
	There exists a function $F:\C^{l}\to \C$,  such that
	\begin{equation*}
		f_{k}(\mathbf{z})=\frac{\partial F}{\partial \overline{z}_{k}}(\mathbf{z})+\overline{\frac{\partial F }{\partial z_{k}}}(\mathbf{z}),\qquad k=1\ldots,l.
	\end{equation*}
\end{thmx}

\renewcommand\thethmx{(H4)}
\begin{thmx}\label{H4}
There exist positive constants $\sigma_{1},\ldots,\sigma_{l}$ such that for any $\mathbf{z}\in \mathbb{C}^{l}$
\begin{equation*}
\mathrm{Im}\sum_{k=1}^{l}\sigma_{k}f_{k}(\mathbf{z})\overline{z}_{k}=0 .
	\end{equation*}	
\end{thmx}

\renewcommand\thethmx{(H5)}
\begin{thmx}\label{H5}
	Function $F$ is homogeneous of degree 3, that is, for any $\mathbf{z}\in \mathbb{C}^{l}$ and $\lambda >0$,
	\begin{equation*}
		F(\lambda \mathbf{z})=\lambda^{3}F(\mathbf{z}).
	\end{equation*}
\end{thmx}

\renewcommand\thethmx{(H6)}
\begin{thmx}\label{H6}
	There holds
	\begin{equation*}
		\left|\mathrm{Re}\int_{\R^{n}} F(\ub)\;dx\right|\leq \int_{\R^{n}} F(\!\!\big\bracevert\!\! \mathbf{u}\!\!\big\bracevert\!\!)\;dx.
	\end{equation*}
\end{thmx}

\renewcommand\thethmx{(H7)}
\begin{thmx}\label{H7}
	Function $F$ is real valued on $\R^l$, that is, if $(y_{1},\ldots,y_{l})\in \R^{l}$ then
	\begin{equation*}
		F(y_{1},\ldots,y_{l})\in \R.
	\end{equation*}
	Moreover, functions	$f_k$ are non-negative on the positive cone in $\mathbb{R}^l$, that is, for $y_i\geq0$, $i=1,\ldots,l$,
	\begin{equation*}
		f_{k}(y_{1},\ldots,y_{l})\geq0.
	\end{equation*}	
\end{thmx}

\renewcommand\thethmx{(H8)}
\begin{thmx}\label{H8}
	Function $F$ can be written as the sum $F=F_1+\cdots+F_m$, where $F_s$, $s=1,\ldots, m$ is super-modular on $\R^d_+$, $1\leq d\leq l$ and vanishes on hyperplanes, that is, for any $i,j\in\{1,\ldots,d\}$, $i\neq j$ and $k,h>0$, we have
	\begin{equation*}
		F_s(y+he_i+ke_j)+F_s(y)\geq F_s(y+he_i)+F_s(y+ke_j), \qquad y\in \R^d_+,
	\end{equation*}
	and $F_s(y_1,\ldots,y_d)=0$ if $y_j=0$ for some $j\in\{1,\ldots,d\}$.
\end{thmx}

Since our main interest is in studying \eqref{system1} in the Sobolev space $H^1(\R^n)$, the above assumptions are ``quite natural''. In fact, assumption \ref{H1} and \ref{H2} are enough to prove a well-posedness result. Under assumptions \ref{H3} and \ref{H4} we able to show that \eqref{system1} conserves the charge and the energy (see \eqref{mass} and \eqref{energy} below), which in turn is sufficient to extend the local solutions to global ones (under some restrictions on the dimension). The remaining assumption are enough to establish the existence of nonnegative symmetric ground states (see \cite{NoPa2} for further details).

\begin{remark}
It is easily seen that \eqref{system1J} satisfies \ref{H1}-\ref{H8}, in which case we have
\begin{equation*}\label{Fsys}
f_{1}(z_1,z_2)=2\overline{z}_1z_2, \quad f_{2}(z_1,z_2)=z_1^2, \quad \mathrm{and}\quad F(z_1,z_2)=\overline{z}_1^2z_2.
\end{equation*}
For additional models with quadratic nonlinearities satisfying \textnormal{\ref{H1}-\ref{H8}} we refer the reader to \cite{kivshar2000multi}, \cite{NoPa2}, and \cite{Pastor2}.
\end{remark}

Before proceeding with this introduction let us recall the notion of mass-resonance associated with \eqref{system1}: we say that \eqref{system1} satisfies the mass-resonance condition provided (see \cite[Definition 1.1]{NoPa3})
	\begin{equation}\label{RC}
\mathrm{Im}\sum_{k=1}^{l}\frac{\alpha_{k}}{\gamma_{k}}f_{k}(\zb)\overline{z}_{k}=0, \quad \zb\in \mathbb{C}^{l}.\tag{RC}
\end{equation}
The mass-resonance condition plays a distinguished role in the mathematical analysis of system \eqref{system1}. To give a flavor of this property, let us recall a virial-type identity satisfied by solutions of system \eqref{system1}.   Set $\Sigma=\{\ub\in \mathbf{H}_{x}^{1}; x\ub\in \mathbf{L}^{2} \},$
where $x\ub$ means $(xu_{1},\ldots,xu_{l})$, and define the function
\begin{equation}\label{fuctV}
	V(t)=\sum_{k=1}^{l}\frac{\alpha_{k}^{2}}{\gamma_{k}}\int |x|^{2}|u_{k}(x,t)|^{2}\;dx,
\end{equation}
where $\ub(t)$ is the corresponding solution of  \eqref{system1} with initial data $\ub_{0}\in \Sigma$. Then,   it is not difficult to see that
\begin{equation}\label{V2mwmr}
	\begin{split}
		V''(t)&=2nE(\ub_{0})-2n\sum_{k=1}^l\beta_k\|u_k\|_{L^2}^2+2(4-n)\sum_{k=1}^l\gamma_k\|\nabla u_k\|_{L^2}\\
		&\quad -2\frac{d}{dt}\left[\int|x|^{2}\mathrm{Im}\sum_{k=1}^{l}\frac{\alpha_k}{\gamma_k}f_{k}(\mathbf{u})\overline{u}_{k}\;dx\right],
	\end{split}
\end{equation}
as long as the solution exists.
Assuming that \eqref{RC}  holds, the last term in \eqref{V2mwmr}  disappears. In particular, this identity becomes useful to use a convexity argument and show the existence of solutions with negative energy  that blow-up in finite time in dimensions $4\leq n\leq 6$. Since the proof of scattering usually also uses  virial-type identities, then mass-resonance condition also plays a crucial role in that analysis.

 In \cite{NoPa2}  we studied some aspects of the dynamics of \eqref{system1} such as local and global well-posedness, existence of standing waves, the dichotomy global existence versus blow-up in finite time and the stability/instability of standing waves. There we consider system \eqref{system1} endowed with assumptions \ref{H1}-\ref{H8} but  with \ref{H4} replaced by
\renewcommand\thethmx{(H4*)}
\begin{thmx}\label{H4*}
	For any $ \theta \in \R$ and $\mathbf{z}\in \mathbb{C}^{l}$,
	\begin{equation*}
		\mathrm{Re}\,F\left(e^{i\frac{\alpha_{1}}{\gamma_{1}}\theta  }z_{1},\ldots,e^{i\frac{\alpha_{l}}{\gamma_{l}}\theta  }z_{l}\right)=\mathrm{Re}\,F(\mathbf{z}).
	\end{equation*}
\end{thmx}

\noindent Assumption \ref{H4*} together with \ref{H3}, implies that \eqref{RC} holds (see Lemma 2.9 in \cite{NoPa2}), implying that all results obtained in \cite{NoPa2} were under the assumption of mass-resonance. However, as pointed out in \cite{NoPa3} most of the results present in \cite{NoPa2} also holds with \ref{H4} instead of \eqref{H4*}. In particular  assuming \ref{H3} and \ref{H4} we can establish that the  quantities
 	\begin{equation}\label{mass}
	Q(\ub(t)):=\sum_{k=1}^{l}\frac{\sigma_{k}\alpha_{k}}{2}\|  u_{k}(t)\|_{L^{2}}^{2},
	\end{equation}
	and
	\begin{equation}
\label{energy}
	E_{\boldsymbol{\beta}}(\ub(t)):=\sum_{k=1}^{l}\gamma_{k}\|\nabla u_{k}(t)\|_{L^2}^{2}+\sum_{k=1}^{l}\beta_{k}\|u_{k}(t)\|_{L^2}^{2}
    -2\mathrm{Re}\int F(\ub(t))\;dx, 
\end{equation}	
are conserved along the flow of \eqref{system1}, which  means that, as long as a solution exists, it satisfies
\begin{equation}\label{conserQE}
Q(\ub(t))=Q(\ub_{0}) \qquad\mathrm{and}\qquad
E_{\boldsymbol{\beta}}(\ub(t))=E_{\boldsymbol{\beta}}(\ub_{0}).
\end{equation}

Using these conserved quantities and \ref{H6} we then got  an \textit{a priori} bound for the $L^{2}$ and $H^{1}$-norm of a solution, so the global well-posedness may be established  in $L^{2}(\R^{n})$ and $H^{1}(\R^{n})$, when $1\leq n\leq 3$. To give a more precise statement concerning the global well-posedness  in dimensions  $n=4$ and $n=5$, recall that a \textit{standing wave} for \eqref{system1} is a  solution of the form
\begin{equation*}
u_{k}(x,t)=e^{i\frac{\sigma_{k}}{2}\omega t}\psi_{k}(x),\qquad k=1,\ldots,l,
\end{equation*}
where $\omega\in \R$ and  $\psi_{k}$ are real-valued functions decaying to zero at infinity, satisfying the elliptic system
\begin{equation}\label{systemelip}
\displaystyle -\gamma_{k}\Delta \psi_{k}+\left(\frac{\sigma_{k}\alpha_{k}}{2}\omega+\beta_{k}\right) \psi_{k}=f_{k}(\psib),\qquad k=1,\ldots,l.
\end{equation}
A \textit{ground state} is a solution of \eqref{systemelip} that minimizes the functional
\begin{equation*}\label{FunctionalI}
I(\boldsymbol{\psi})=\frac{1}{2}\left[\sum_{k=1}^{l}\gamma_{k}\|\nabla \psi_{k}\|_{L^2}^{2}+\sum_{k=1}^{l}\left(\frac{\sigma_{k}\alpha_{k}}{2}\omega+\beta_{k}\right)\| \psi_{k}\|_{L^2}^{2}\right]
-\int F(\boldsymbol{\psi})\;dx.
\end{equation*}
 Under our assumptions (see \cite{NoPa2} and \cite{NoPa3}),  if the coefficients $\frac{\sigma_{k}\alpha_{k}}{\gamma_{k}}\omega+\beta_{k}$ are positive then the set of ground states of \eqref{systemelip}, say, $\mathcal{G}_{n}(\omega,\boldsymbol{\beta})$ is nonempty for $1\leq n\leq5$. In addition introducing the functionals
\begin{equation}\label{functionalQ}
\mathcal{Q}(\boldsymbol{\psi})=\sum_{k=1}^{l}\left(\frac{\sigma_{k}\alpha_{k}}{2}\omega+\beta_{k}\right)\| \psi_{k}\|_{L^2}^{2},
\end{equation}
\begin{equation}\label{funclKP6}
K(\boldsymbol{\psi})=\sum_{k=1}^{l}\gamma_{k}\|\nabla \psi_{k}\|_{L^2}^{2}\qquad \mathrm{and}\qquad P(\boldsymbol{\psi})=\int F(\boldsymbol{\psi})\;dx,
\end{equation}
we have the following Gagliardo-Nirenberg-type inequality (see \cite[Corollary 4.12]{NoPa2}),
\begin{equation}\label{GNI}
    P(\ub)\leq C_{n}^{opt}\mathcal{Q}(\ub)^{\frac{6-n}{4}}K(\ub)^{\frac{n}{4}},
\end{equation}
for all functions $\ub\in \mathcal{P}:=\{\psib\in \mathbf{H}_{x}^{1}(\R^n);\, P(\psib)>0\}$, with the optimal constant $C_{n}^{opt}$ given by
\begin{equation}\label{bestCn}
    C_{n}^{opt}:=\frac{2(6-n)^{\frac{n-4}{4}}}{n^{\frac{n}{4}}}\frac{1}{\mathcal{Q}(\psib)^{\frac{1}{2}}}.
\end{equation}
In \eqref{bestCn}, $\psib$ is any function in $\mathcal{G}_{n}(\omega,\boldsymbol{\beta})$.

As a consequence of \eqref{GNI} we may prove (see \cite{NoPa3}, also see \cite[Theorem 5.2]{NoPa2}) that, if $\ub_0\in \mathbf{H}_{x}^{1}(\R^4)$ satisfies $Q(\ub_0)<Q(\psib)$, where $\psib$ is any function in $\mathcal{G}_{4}(1,\mathbf{0})$ then the corresponding solution of \eqref{system1} may be extended globally in $\mathbf{H}_{x}^{1}(\R^4)$. In dimension $n=5$ we established the following.\\

\noindent {\bf Theorem A.} \textit{Let $n=5$ and assume \textnormal{\ref{H1}-\ref{H8}}. 
	Assume $\mathbf{u}_{0}\in \mathbf{H}_{x}^{1}$ and let $\mathbf{u}$ be the corresponding solution of system \eqref{system1}. Let  $\boldsymbol{\psi}\in \mathcal{G}_{5}(1,\mathbf{0})$ be a ground state. If
	\begin{equation}\label{desEQgs1}
	Q(\mathbf{u}_{0})	E_{\boldsymbol{\beta}}(\mathbf{u}_{0})<Q(\boldsymbol{\psi})\E(\boldsymbol{\psi}),
	\end{equation}
	where $\E$ is the energy defined in \eqref{energy} with $\boldsymbol{\beta}=\boldsymbol{0}$
	and
	\begin{equation}\label{desQKgs1}
	Q(\mathbf{u}_{0})K(\mathbf{u}_{0})<Q(\boldsymbol{\psi})K(\boldsymbol{\psi}),  
	\end{equation}
	then $\ub$ is global  in $\mathbf{H}_{x}^{1}$ and satisfies
	$$
	Q(\ub_0)\sup_{t\in\R}K(\ub(t))<Q(\boldsymbol{\psi})K(\boldsymbol{\psi}).
	$$
In addition if we assume \textnormal{\ref{H4*}} instead of \textnormal{\ref{H4}} and that $\ub_0$ is radial then  $\ub$ scatters. 
}\\

For the proof of Theorem A we refer the reader to \cite[Theorem 5.2]{NoPa2} and \cite[Theorem 1.1]{NoPa4}. As we already said, under assumption \ref{H4*}, system \eqref{system1} satisfies \eqref{RC}. Consequently, the scattering in Theorem A is proved only in the case of mass-resonance. 

Our main purpose in this paper is to prove that scattering still holds only under assumptions \eqref{desEQgs1} and \eqref{desQKgs1}, that is, we are able to drop the assumptions of mass-resonance and radial symmetry as well. More precisely, our main theorem reads as follows.

\begin{theorem}\label{thm:critscatt} Let $n=5$ and assume \textnormal{\ref{H1}-\ref{H8}}. 
	Suppose $\mathbf{u}_{0}\in \mathbf{H}_{x}^{1}$ and let $\mathbf{u}$ be the corresponding solution of system \eqref{system1}. Let  $\boldsymbol{\psi}\in \mathcal{G}_{5}(1,\mathbf{0})$ be a ground state. Assume
\begin{equation}\label{desEQgs}
	Q(\mathbf{u}_{0})	E_{\boldsymbol{\beta}}(\mathbf{u}_{0})<Q(\boldsymbol{\psi})\E(\boldsymbol{\psi}),
\end{equation}
and
\begin{equation}\label{desQKgs}
	Q(\mathbf{u}_{0})K(\mathbf{u}_{0})<Q(\boldsymbol{\psi})K(\boldsymbol{\psi}).
\end{equation}
Assume also that 
\begin{equation}\label{smallreson}
    \left|\frac{\alpha_{k}}{\gamma_{k}}-\sigma_{k}\right|<\varepsilon_{1}, \qquad k=1,\ldots,l,
\end{equation} 
for some $\varepsilon_{1}>0$ small enough. Then, $\ub$ is global and scatters forward and backward in time, that is,  there exist $u^{\pm}_{k}\in H^{1}(\R^{5})$ such that 
	\begin{equation*}
	\lim_{t\to +\infty}\|u_{k}(t)-U_{k}(t)u_{k}^{+}\|_{H^{1}}=0, \qquad  k=1,\ldots,l
	\end{equation*}
and
	\begin{equation*}
\lim_{t\to -\infty}\|u_{k}(t)-U_{k}(t)u_{k}^{-}\|_{H^{1}}=0, \qquad  k=1,\ldots,l.
\end{equation*}
\end{theorem}

  Theorem \ref{thm:critscatt} states that if the masses of the system are sufficiently close to the constant appearing in assumption \ref{H4} then scattering holds. Of course, if the constants $\frac{\alpha_{k}}{\gamma_k}$ and $\sigma_{k}$  coincide, we are in the mass-resonance case and the scattering result still holds.

Our idea  to prove Theorem \ref{thm:critscatt} is to apply the recent theory introduced in \cite{dodsonmurphy2018}, where the authors have shown the scattering in $H^1(\R^n)$, with non-radial data,  for the Schr\"odinger equation
\begin{equation*}
\displaystyle i\partial_{t}u+\Delta u=-|u|^{\frac{4}{n-4}}u,
\end{equation*}
 in dimensions $n\geq 3$ using a simple proof that avoids the concentration-compactness. Instead, their analysis is based on an interaction Morawetz inequality. In the same direction, scattering for some Schr\"odinger-type system have appeared, for instance, in \cite{ardila}, \cite{mengxu2020},\cite{wang2019Sacttering}.

Before ending this introduction, let us recall the scattering results for system \eqref{system1J}. First note that \eqref{system1J} satisfies the mass-resonance condition if and only if $\ka=1/2$. In dimension $n=5$, under similar assumption as in Theorem A,  in \cite{hamano2018global} the author established the scattering of radially symmetric solutions in $H^1(\R^5)$ in the case $\ka=1/2$; the mass-resonance assumption was dropped in \cite{hamano2019scattering}. In both cases, the authors used the concentration-compactness and rigidity method introduced in \cite{KenigMarleScattering}. More recently, by using the ideas introduced in \cite{dodsonmurphy2018}, the assumption of radial symmetry was dropped in  \cite{mengxu2020} and the assumption of mass-resonance was dropped in  \cite{wang2019Sacttering}. Thus our result can also be viewed as an extension of the results in \cite{wang2019Sacttering} to a more general quadratic-type NLS systems.

Some results have also appeared in the critical case $n=4$. Indeed, scattering below the ground states in $L^2(\R^4)$ was established in \cite{inui2019scattering}. More precisely, the authors established that scattering holds for any initial data below the ground state  if $\ka=1/2$ and for radial initial data below the ground state  if $\ka\neq1/2$.

This work is organized as follows. In section \ref{sec.prel} we recall some notation and preliminary lemmas that will be needed throughout the paper.
In section \ref{sec.scattcrit}, using the strategy introduced in \cite{dodsonmurphy2018}, we prove a scattering criterion for \eqref{system1}. Finally, section \ref{sec.proofscatt} is devoted to showing Theorem \ref{thm:critscatt}, by using a  Morawetz estimate and the ground state solutions.



\section{Preliminaries}\label{sec.prel}
In this section we introduce some notations, review some useful estimates and give  consequences of our assumptions.
\subsection{Notation}
 We use $C$ to denote several positive constants that may vary line-by-line. If $a$ and $b$ are two positive constants, by $a\lesssim b$ we mean there is a constant $C$ such that $a\leq Cb$.
Given any set $A$, by  $\mathbf{A}$ (or $A^{l}$) we denote  the product  $\displaystyle A\times \cdots \times A $ ($l$ times). In particular, if $A$ is a Banach space then $\mathbf{A}$ is also a Banach space with the standard norm given by the sum.  For a number $z\in\mathbb{C}$, $\mathrm{Re}\,z$  and $\mathrm{Im}\,z$ represents its real and imaginary parts. Also, $\overline{z}$ denotes its complex conjugate. We set $\big\bracevert\!\! \mathbf{z}\!\!\big\bracevert$  for the vector $(|z_{1}|,\ldots,|z_{l}|)$. This is not to be confused with $|\mathbf{z}|=\sqrt{z_1^2+\ldots+z_l^2}$ which stands for the standard norm of the vector $\mathbf{z}$ in $\mathbb{C}^l$.

The space
 $L^{p}=L^{p}(\R^{n})$, $1\leq p\leq \infty$, stands for the standard Lebesgue spaces. By $W^{s,p}=W^{s,p}(\R^{n})$, $1\leq p\leq \infty$, $s\in\R$, we denote the usual Sobolev spaces. In the case $p=2$, we use the standard notation $H^s=W^{s,2}$. Thus, $\mathbf{H}^1=\mathbf{H}^1(\R^n)$ denotes the Sobolev space $H^1\times \cdots \times H^1$.

Given a time interval $I$ and  a Banach space $X$, $L^{p}(I;X)$ represents the $L^{p}$ space of $X$-valued functions defined on $I$ endowed with the norm
$$
\|f\|_{L^{p}_{t}X}=\left(\int_I \|f(t)\|_X^p dt \right)^{\frac{1}{p}}.
$$
In the case $X=L^q(\R^n)$, we use the notation $L^{p}_{t}L^{q}_{x}(I\times\R^n)$. If no confusion will be caused we denote  $L^{p}_{t}L^{q}_{x}(I\times\R^n)$ simply by  $L^{p}_{t}L^{q}_{x}$ and its norm by $\|\cdot\|_{L^{p}_{t}L^{q}_{x}}$. Also, when $p=q$ we will use $L^{p}_{tx}$ instead of $L^{p}_{t}L^{q}_{x}$. If necessary we use subscript to indicate which variable the spaces are taken, for instance, $\mathbf{H}^1_x$ represents the Sobolev space with respect to the variable $x$.

\subsection{Some useful estimates} 

Here and throughout the paper we assume that \ref{H1}-\ref{H8} hold.
Let $U_{k}(t)$ denote the Schr\"odinger evolution group defined by $\displaystyle U_{k}(t)=e^{i\frac{t}{\alpha_{k}}(\gamma_{k}\Delta-\beta_{k})}$, $ k=1\ldots, l$. In view of Duhamel's principle the Cauchy problem \eqref{system1} can be written as the following  system of integral equations,
\begin{equation}\label{system2}
\begin{cases}
u_{k}(t)= \displaystyle U_{k}(t)u_{k0}+i\int_{0}^{t}U_{k}(t-t') \frac{1}{\alpha_{k}}f_{k}(\mathbf{u})\;dt',\\
(u_{1}(x,0),\ldots,u_{l}(x,0))=(u_{10},\ldots,u_{l0})=:\ub_0.
\end{cases}
\end{equation}
Thus, in what follows, by a solution of \eqref{system1} we mean a solution of \eqref{system2}.

Let us start be recalling the following dispersive estimate.

\begin{lemma}\label{dispest}
	If $2\leq p\leq\infty$ and $t\neq 0$, then 
	\begin{equation*}
	\|U_{k}(t)f\|_{L_{x}^{p}(\R^{n})}\lesssim |t|^{-n\left(\frac{1}{2}-\frac{1}{p}\right)}\|f\|_{L_{x}^{p'}(\R^{n})}, \qquad \mbox{for all $f\in L_{x}^{p'}(\R^{n})$.}
	\end{equation*}
\end{lemma}
\begin{proof}
	See Proposition 2.2.3 in \cite{Cazenave}.
\end{proof}

Before proceeding we recall the definition of admissible pair in dimension $n=5$.

\begin{definition}
We say that $(q,r)$ is an admissible  pair if
\begin{equation*}
    \frac{2}{q}+\frac{5}{r}=\frac{5}{2},
\end{equation*}
where $2\leq r\leq \frac{10}{3}$.
\end{definition}

We next recall  the well known Strichartz inequalities. 

\begin{proposition}[Strichartz's inequalities]\label{stricha}
The following inequalities hold.
\begin{itemize}
	\item[(i)] Let $(q,r)$ be an admissible pair. Then,
	$$
	\|U_{k}(t)f\|_{L_{t}^{q}L_{x}^{r}(\R\times\R^{5})}\lesssim\|f\|_{L_{x}^{2}(\R^{5})}.
	$$
	\item[(ii)] Let $I$ be an interval  and $t_0\in \overline{I}$. Let $(q_1,r_1)$ and $(q_2,r_2)$ be two admissible pairs. Then, 
	$$
	\left\| \int_{t_0}^tU_{k}(t-s)f(\cdot,s)ds \right\|_{L_{t}^{q_{1}}L_{x}^{r_{1}}(I\times\R^{5})} \lesssim \|f\|_{L_{t}^{q_{2}'}L_{x}^{r_{2}'}(I\times\R^{5})}
	$$
	and
	$$
		\left\| \int_{a}^bU_{k}(t-s)f(\cdot,s)ds \right\|_{L_{t}^{q_{1}}L_{x}^{r_{1}}(\R\times\R^{5})} \lesssim \|f\|_{L_{t}^{q_{2}'}L_{x}^{r_{2}'}([a,b]\times\R^{5})},
	$$
	where $q_2'$ and $r_2'$ are the H\"older conjugates of $q_2$ and $r_2$, respectively.
	\item[(iii)] Let $I$ be an interval  and $t_0\in \overline{I}$. Then
		$$
	\left\| \int_{t_0}^tU_{k}(t-s)f(\cdot,s)ds \right\|_{L_{t}^{6}L_{x}^{3}(I\times\R^{5})} \lesssim \|f\|_{L_{t}^{3}L_{x}^{\frac{3}{2}}(I\times\R^{5})}.
	$$
\end{itemize}
\end{proposition}
\begin{proof}
	For (i) and (ii) see, for instance, Theorem 2.3.3 in \cite{Cazenave}. For (iii) note that  $(6,3)$ is not an admissible pair. However, $(6,3)$ and $(\frac{3}{2},3)$ are $\frac{5}{2}$-acceptable pairs. Recall that a pair $(q,r)$ is said to be $\sigma$-acceptable if
	$$
	\frac{1}{q}<2\sigma\left(\frac{1}{2}-\frac{1}{r}\right).
	$$
	Hence the result follows as an application of Proposition 6.2 in \cite{foschi}.
\end{proof}

The next lemma provides some consequences of our assumptions concerning the nonlinearities.

 \begin{lemma}\label{estdiffk} Assume that  \textnormal{\ref{H1}-\ref{H5}}   hold. 
 
 \begin{enumerate}
     \item[(i)] For all $\mathbf{z}\in\mathbb{C}^l$  we have
\begin{equation*}
\left|f_{k}(\mathbf{z})\right|\lesssim \sum_{j=1}^{l}|z_{j}|^{2}, \qquad k=1\ldots, l
\end{equation*}
and
\begin{equation*}
|\mathrm{Re}\,F(\mathbf{z})|\lesssim  \sum_{j=1}^{l}|z_{j}|^{3}.
\end{equation*}
\item[(ii)] We have
\begin{equation*}
\mathrm{Re}\sum_{k=1}^{l}f_{k}(\mathbf{u})\nabla \overline{u}_{k}=\mathrm{Re}\,[\nabla F(\mathbf{u})]
\end{equation*}
and
\begin{equation*}
\mathrm{Re}\sum_{k=1}^{l}f_{k}(\mathbf{u})\overline{u}_{k}=\mathrm{Re}\,[3F(\mathbf{u})].
\end{equation*}
\item[(iii)] Let $1< p,q,r< \infty$ be such that $\frac{1}{r}=\frac{1}{p}+\frac{1}{q}$.  Then, for $k=1,\ldots,l$,
\begin{equation}\label{lei11}
	\|\nabla f_{k}(\mathbf{u})\|_{L^{r}}\lesssim \|\mathbf{u}\|_{\mathbf{L}_{x}^{p}}\| \nabla\mathbf{u}\|_{\mathbf{L}_{x}^{q}(\R^n)}
\end{equation}
and 
\begin{equation}\label{lei12}
\|f_{k}(\mathbf{u})\|_{{W}^{\frac{1}{2},r}}\lesssim \|\mathbf{u}\|_{\mathbf{L}_{x}^{p}}\| \mathbf{u}\|_{\mathbf{W}_{x}^{\frac{1}{2},q}(\R^n)}.
\end{equation}
 \end{enumerate}
\end{lemma}
\begin{proof}
For (i) and (ii) see  Corollary 2.3 and Lemmas 2.10 and 2.11 in  \cite{NoPa2}. Part (iii) is a consequence of \ref{H2} and the Leibniz rule (see Proposition 5.1 in \cite{mtaylor} and Corollary 2.5 in \cite{NoPa2}).
\end{proof}

The next result is a refinement of inequality \eqref{GNI}.

\begin{lemma}\label{lemGNIref}
	Let $\psib\in \mathcal{G}_5(\omega,\betab)$.
For any $\ub \in \mathbf{H}_{x}^1$ and $\xi \in \R^{5}$ we have
\begin{equation*}
    \left|\mathrm{Re}\int_{\R^{5}}F(\ub)\;dx\right|\leq \frac{2}{5}\left[\frac{\Q(\ub)K(\ub)}{\Q(\psib)K(\psib)}\right]^{\frac{1}{4}}K(\ub^{\xi}),
\end{equation*}
where
\begin{equation}\label{Gaugetrans}
	\ub^\xi(x)=\left(e^{i\frac{\alpha_{1}}{\gamma_1}x\cdot\xi}u_1(x),\ldots,e^{i\frac{\alpha_{l}}{\gamma_l}x\cdot\xi}u_l(x)\right).
\end{equation}
\end{lemma}
\begin{proof}
First note if $\boldsymbol{\psi}$ is a solution of \eqref{systemelip} (with $n=5$) then
		\begin{equation*}
			P(\boldsymbol{\psi})=2I(\boldsymbol{\psi}), \quad K(\boldsymbol{\psi})=5I(\boldsymbol{\psi}), \quad \mbox{and}\quad \mathcal{Q}(\boldsymbol{\psi})=I(\boldsymbol{\psi}).
		\end{equation*}
Thus, since $K(\psib)=5\Q(\psib)$  the best constant \eqref{bestCn} can be expressed as
\begin{equation}\label{C5}
    C_{5}^{opt}=\frac{2}{5}\frac{1}{\Q(\psib)^{\frac{1}{4}}K(\psib)^{\frac{1}{4}}}. 
\end{equation}
Hence, from \eqref{GNI},
\begin{equation}\label{C5.1}
	 |P(\!\!\big\bracevert\!\! \mathbf{u}\!\!\big\bracevert\!\!)|\leq \frac{2}{5}\left[\frac{\Q(\!\!\big\bracevert\!\!\ub\!\!\big\bracevert\!\!)K(\!\!\big\bracevert\!\!\ub\!\!\big\bracevert\!\!)}{\Q(\psib)K(\psib)}\right]^{\frac{1}{4}}K(\!\!\big\bracevert\!\!\ub\!\!\big\bracevert\!\!)
	 \leq \frac{2}{5}\left[\frac{\Q(\ub)K(\ub)}{\Q(\psib)K(\psib)}\right]^{\frac{1}{4}}K(\ub).
\end{equation}
By replacing $\ub$ by $\ub^\xi$ in \eqref{C5.1} and using that $P(\!\!\big\bracevert\!\! \ub\!\!\big\bracevert\!\!)=P(\!\!\big\bracevert\!\! \ub^\xi\!\!\big\bracevert\!\!)$ we obtain
 \begin{equation*}
     |P(\!\!\big\bracevert\!\! \mathbf{u}\!\!\big\bracevert\!\!)|\leq\frac{2}{5}\left[\frac{\Q(\ub^\xi)K(\ub^{\xi})}{\Q(\psib)K(\psib)}\right]^{\frac{1}{4}}K(\ub^{\xi}).
 \end{equation*}
Now, by taking the infimum  over all $\xi\in\R^5$ we get
 \begin{equation*}
     \begin{split}
      |P(\!\!\big\bracevert\!\! \mathbf{u}\!\!\big\bracevert\!\!)| &\leq \frac{2}{5}  \inf_{\xi\in \R^{5}} \left\{ \left[\frac{\Q(\ub^\xi)K(\ub^{\xi})}{\Q(\psib)K(\psib)}\right]^{\frac{1}{4}}K(\ub^{\xi})\right\}\\
       &\leq\frac{2}{5}  \inf_{\xi\in \R^{5}} \left\{ \left[\frac{\Q(\ub^\xi)K(\ub^{\xi})}{\Q(\psib)K(\psib)}\right]^{\frac{1}{4}}\right\}\times \inf_{\xi\in \R^{5}} K(\ub^{\xi})\\
      &\leq \frac{2}{5}\left[\frac{\Q(\ub)K(\ub)}{\Q(\psib)K(\psib)}\right]^{\frac{1}{4}}   K(\ub^{\xi}).
         \end{split}
 \end{equation*}
The lemma then follows from hypothesis \ref{H6}.
\end{proof}

 \subsection{Coercivity lemmas}\label{subseccoerxci}

We finish this section with some coercivity lemmas.  Note that in the case $\omega=1$ and $\boldsymbol{\beta}=\mathbf{0}$ the functionals $Q$ and $\mathcal{Q}$ coincides.
\begin{lemma}[Coercivity I]\label{thm:lemcoerI} Let $n=5$. 
 	Assume $\mathbf{u}_{0}\in \mathbf{H}_{x}^{1}$  and let $\mathbf{u}$ be the corresponding solution of system \eqref{system1} with maximal existence interval $I$. Let  $\boldsymbol{\psi}\in \mathcal{G}_{5}(1,\mathbf{0})$ be a ground state. If 
 \begin{equation}\label{b-1}
   Q(\mathbf{u}_{0})	E_{\boldsymbol{\beta}}(\mathbf{u}_{0})<(1-\widetilde{\delta})Q(\boldsymbol{\psi})\E(\boldsymbol{\psi})
\end{equation}
and
\begin{equation}\label{b-2}
  Q(\mathbf{u}_{0})K(\mathbf{u}_{0})\leq Q(\boldsymbol{\psi})K(\boldsymbol{\psi}),  
\end{equation}
then there exist $\delta'>0$, depending on $\widetilde{\delta}$, such that
\begin{equation*}
  Q(\mathbf{u}_{0})K(\mathbf{u}(t))<(1-\delta') Q(\boldsymbol{\psi})K(\boldsymbol{\psi}),  
\end{equation*}
for all $t\in I$. In particular, $I=\R$ and $\ub$ is uniformly bounded in $\mathbf{H}_{x}^{1}$. 
 \end{lemma}
\begin{proof}
See Lemma 2.5 in \cite{NoPa4}.
\end{proof}

 Next we will prove a coercivity lemma on balls. We start with the following.

\begin{corollary}\label{remarksup}
	Under the assumptions of Theorem \ref{thm:critscatt}, there exists
	$\delta>0$ such that $$\sup_{t\in \R}Q(\ub_{0})K(\ub(t))<(1-3\delta)^{4}Q(\psib)K(\psib).$$ 
\end{corollary}
\begin{proof}
This is an immediate consequence of Lemma \ref{thm:lemcoerI}.
\end{proof}

Now, if $\varepsilon>0$ is a small constant, let us introduce a radial function $\chi\in C_0^\infty(\R^5)$, decreasing on $|x|$, such that $0\leq \chi\leq1$ and
\begin{equation}\label{defchi}
	\chi(x)=\left\{\begin{array}{cl}
		1,& |x|\leq 1-\varepsilon,\\
		0,& |x|\geq 1.
	\end{array}\right.
\end{equation}

\begin{lemma}[Coercivity on balls]\label{lemcoerbal} Let $\delta>0$ be as in Corollary \ref{remarksup}. Then, there exists a sufficiently large $R=R(\delta,Q(\ub),\psib)>0$  such that for all $s\in \R^{5}$,
	\begin{equation*}
		\sup_{t\in \R}Q\left(\chi\left(\frac{\cdot-s}{R}\right)\ub(t)\right)K\left(\chi\left(\frac{\cdot-s}{R}\right)\ub(t)\right)<(1-2\delta)^{4}Q(\psib)K(\psib).
	\end{equation*}
	In particular, by Lemma \ref{lemGNIref}, there exists $R=R(\delta,Q(\ub),\psib)>0$ (probably larger) so that for any $\xi \in \R^{5}$
	\begin{equation*}
		\begin{split}
			\left|\int_{\R^{5}}\mathrm{Re}\, F\left(\chi\left(\frac{x-s}{R}\right)\mathbf{u}(x)\right)\;dx\right|
			&\leq \frac{2}{5}(1-\delta)\int_{\R^{5}}\chi^{2}\left(\frac{x-s}{R}\right)\sum_{k=1}^{l} \gamma_{k}|\nabla u_{k}^{\xi}|^{2}\;dx,
		\end{split}
	\end{equation*}
	uniformly for $t\in \R$. 
\end{lemma}
\begin{proof}
	First we observe that  integrating by parts and using that $\chi$ has  compact support, we obtain
	\begin{equation}\label{intequal}
		\int \chi^{2} |\nabla u_{k}|^{2}\; dx=\int  |\nabla (\chi u_{k})|^{2}\;dx +\int \chi\Delta (\chi)|u_{k}|^{2}\;dx. 
	\end{equation}
	Multiplying by $\gamma_{k}$ and summing over $k$ we have
	\begin{equation}\label{identK}
		\int \chi^{2}\sum_{k=1}^{l} \gamma_{k}|\nabla u_{k}|^{2}\; dx= K(\chi \ub) +\int \chi\Delta (\chi)\sum_{k=1}^{l}\gamma_{k}|u_{k}|^{2}\;dx.
	\end{equation}
By replacing $\chi$ by $\chi\left(\frac{\cdot-s}{R}\right)$ in \eqref{identK} and using that  $\left|\Delta \chi\left(\frac{x-s}{R}\right)\right|\leq \frac{C}{R^{2}}$, we arrive to 
	
	\begin{equation}\label{b3}
		\begin{split}
			K\left(\chi\left(\frac{x-s}{R}\right)\ub\right) &= \int \chi^{2}\left(\frac{x-s}{R}\right)\sum_{k=1}^{l} \gamma_{k}|\nabla u_{k}|^{2}\; dx -\int \chi\left(\frac{x-s}{R}\right)\Delta \left(\chi\left(\frac{x-s}{R}\right)\right)\sum_{k=1}^{l}\gamma_{k}|u_{k}|^{2}\;dx\\
			&\leq K(\ub)+\frac{C}{R^{2}}Q(\ub_{0}). 
		\end{split}
	\end{equation}
	Thus by noting that  $Q\left(\chi\left(\frac{x-s}{R}\right)\ub(t)\right)\leq Q(\ub(t))=Q(\ub_{0})$, from \eqref{b3} and Corollary \ref{remarksup}  we deduce 
	
	\begin{equation*}
		\begin{split}
			Q\left(\chi\left(\frac{x-s}{R}\right)\ub(t)\right)K\left(\chi\left(\frac{x-s}{R}\right)\ub(t)\right)&\leq Q(\ub_{0})K\left(\chi\left(\frac{x-s}{R}\right)\ub(t)\right)\\
			&\leq Q(\ub_{0})\left[K(\ub(t))+\frac{C}{R^{2}}Q(\ub_{0})\right]\\
			&\leq Q(\ub_{0})K(\ub(t))+\frac{C}{R^{2}}Q(\ub_{0})^{2}\\
			&<(1-3\delta)^{4}Q(\psib)K(\psib)+\frac{C}{R^{2}}Q(\ub_{0})^{2}.
		\end{split}
	\end{equation*}
	Choosing $R$ sufficiently large  such that $\frac{C}{R^{2}}Q(\ub_{0})^{2}<[(1-2\delta)^{4}-(1-3\delta)^{4}]
	Q(\psib)K(\psib)$, we obtain the desired.
	
	The second conclusion of the lemma is a direct consequence of Lemma \ref{lemGNIref} combined with  \eqref{identK} and a similar argument as above. The result is thus established.
\end{proof}

 \section{Scattering criterion}\label{sec.scattcrit}
 This section is devoted to prove the scattering criterion we use to prove Theorem \ref{thm:critscatt}.
 We will be concerned only with scattering forward in time; in a similar fashion we may also prove the scattering backward in time.

\begin{lemma}\label{scat1}
 Let $\ub$ be a global solution of \eqref{system1} satisfying
 \begin{equation}\label{unifbond}
   \sup_{t\in\R}\|\ub(t)\|_{\Hb^{1}_{x}}\lesssim E_0,  
 \end{equation}
 for some constant $E_0>0$. If
 \begin{equation}\label{sccrit}
 		\|\ub\|_{\Lb_t^{6}\Lb_x^{3}(\R\times\R^5)}<\infty
 \end{equation}
then $\ub$ scatters.
 \end{lemma}
\begin{proof}
We will prove that $\ub$ scatters forward in time. We first claim that
\begin{equation}
	\|\ub\|_{\Lb^{\frac{12}{5}}_{t}\Wb^{1,3}_{x}(\R\times\R^5)}<\infty.
\end{equation}
Indeed, we decompose $\R=\bigcup_{j=1}^JI_j$ such that
\begin{equation*}
		\|\ub\|_{\Lb_t^{6}\Lb_x^{3}(I_j\times\R^5)}<\delta,
\end{equation*}
where $\delta>0$ is a small constant to be determined later. Thus, since $\left(\frac{12}{5},3\right)$ is admissible, Strichartz estimate, \eqref{lei11} and H\"older's inequality give
\begin{equation}\label{w13es}
\begin{split}
		\|u_k\|_{L^{\frac{12}{5}}_{t}W^{1,3}_{x}(I_j\times\R^5)}&\lesssim \|u_{k0}\|_{H^{1}_{x}}+\|f_k(\ub)\|_{L^{\frac{12}{7}}_{t}W^{1,\frac{3}{2}}_{x}(I_j\times\R^5)}\\
		&\lesssim \|u_{k0}\|_{H^{1}_{x}}+	\|\ub\|_{\Lb_t^{6}\Lb_x^{3}(I_j\times\R^5)} \|\ub\|_{\Lb^{\frac{12}{5}}_{t}\Wb^{1,3}_{x}(I_j\times\R^5)}\\
		& \lesssim \|u_{k0}\|_{H^{1}_{x}}+ \delta \|\ub\|_{\Lb^{\frac{12}{5}}_{t}\Wb^{1,3}_{x}(I_j\times\R^5)}.
\end{split}
\end{equation}
By summing over $k$ and choosing $\delta$ small enough we see that $\|\ub\|_{\Lb^{\frac{12}{5}}_{t}\Wb^{1,3}_{x}(I_j\times\R^5)}\lesssim E_0$. A summation on $j$ then yields the claim.

Now we define
\begin{equation*}
	u_{k}^{+}:=u_{k0}+i\int_{0}^{+\infty}U_{k}(-s) \frac{1}{\alpha_{k}}f_{k}(\mathbf{u})\;ds,
\end{equation*}
Since $(\infty,2)$ is admissible it is not difficult to check that the above integral converges in $H^1(\R^5)$. A straightforward calculation gives,
\begin{equation*}
	\begin{split}
		u_{k}(t)-U_{k}(t)u_{k}^{+}=i\int_{t}^{\infty}U_{k}(t-s)\frac{1}{\alpha_{k}}f_{k}(\ub(s))\;ds.
	\end{split}
\end{equation*}
Hence,  since $(\infty,2)$ is admissible, as in \eqref{w13es} we deduce
\begin{equation*}
	\|u(t)-U_{k}(t)u_{k}^{+}\|_{H^{1}_{x}(\R^{5})}\lesssim \|\ub\|_{\Lb_t^6\Lb_x^3([t,+\infty)\times\R^5)}\|\ub\|_{\Lb_t^{\frac{12}{5}}\mathbf{W}_x^{1,3}([t,+\infty)\times\R^5)},
\end{equation*}
By taking the limit as $t\to+\infty$  in the last inequality we see that right-hand side goes to zero and the proof is completed. 
\end{proof}

The following result is an adapted version of Theorem 3.1 in \cite{dodsonmurphy2018}, see also \cite[Proposition 3.1]{mengxu2020}.

\begin{theorem}\label{scatteringcriteio}
Let $\ub$ be a global solution of \eqref{system1} satisfying
\begin{equation}\label{gloestH1}
	\sup_{t\in\R}\|\ub(t)\|_{\Hb^{1}_{x}}\lesssim E_0,
\end{equation}
for some constant $E_0>0$. Suppose that for any $a\in\R$, there is $t_0\in(a,a+T_0)$ such that
\begin{equation}\label{t_0T_0}
	\|\ub\|_{\Lb_t^{6}\Lb_x^{3}([t_0-\varepsilon^{-\alpha},t_0]\times\R^5)}\lesssim \varepsilon^{\mu},
\end{equation}
where $\varepsilon=\varepsilon(E_0)>0$ is sufficiently small, $T_0=T_0(\varepsilon,E_0)$ is sufficiently large and  $0<\alpha,\mu\leq1$. Then $\ub$ scatters forward in time.
\end{theorem}

Before establishing Theorem \ref{scatteringcriteio} we prove the following.

\begin{lemma}\label{firstlemma}
	Let $\ub$ be as in Theorem \ref{scatteringcriteio}. If \eqref{t_0T_0} holds then for any $a\in\R$, there is $t_0\in(a,a+T_0)$ such that
\begin{equation}\label{esL6L3}
	\left\|\int_0^{t_0}U_k(t-s)f_k(\ub)ds \right\|_{L_t^{6}L_x^{3}([t_0,\infty)\times\R^5)}\lesssim \varepsilon^\nu, \quad k=1,\ldots,l,
\end{equation}
where $\nu=\min\{\mu,\frac{\alpha}{6}\}$.
\end{lemma}
\begin{proof}
Given $a\in\R$ we take $t_0$ such that \eqref{t_0T_0} holds. Write
\[
\begin{split}
\int_0^{t_0}U_k(t-s)f_k(\ub)ds& =\int_0^{t_0-\varepsilon^{-\alpha}}U_k(t-s)f_k(\ub)ds +\int_{t_0-\varepsilon^{-\alpha}}^{t_0}U_k(t-s)f_k(\ub)ds \\
&=:F_1+F_2.
\end{split}
\]
Let us first estimate $F_2$. We claim that 
\begin{equation*}
	\|\ub\|_{\Lb^2_t\Wb_x^{\frac{1}{2},\frac{10}{3}}([t_0-\varepsilon^{-\alpha},t_0]\times\R^5)}\lesssim E_0.
\end{equation*}
Indeed, since $(2,\frac{10}{3})$ and $(3,\frac{30}{11})$ are admissible pairs, \eqref{lei12} and H\"older's inequality yield
\[
\begin{split}
\|u_k\|_{L^2_tW_x^{\frac{1}{2},\frac{10}{3}}([t_0-\varepsilon^{-\alpha},t_0]\times\R^5)}& \lesssim \|u_{k0}\|_{W^{\frac{1}{2},2}_{x}}+\|f_k(\ub)\|_{L^{\frac{3}{2}}_tW_x^{\frac{1}{2},\frac{30}{19}}([t_0-\varepsilon^{-\alpha},t_0]\times\R^5)}\\
& \lesssim \|u_{k0}\|_{H^{1}_{x}}+ \|\ub\|_{\Lb_t^{6}\Lb_x^{3}([t_0-\varepsilon^{-\alpha},t_0]\times\R^5)} 	\|\ub\|_{\Lb^2_t\Wb_x^{\frac{1}{2},\frac{10}{3}}([t_0-\varepsilon^{-\alpha},t_0]\times\R^5)}\\
& \lesssim \|u_{k0}\|_{H^{1}_{x}}+ \varepsilon^{\mu}	\|\ub\|_{\Lb^2_t\Wb_x^{\frac{1}{2},\frac{10}{3}}([t_0-\varepsilon^{-\alpha},t_0]\times\R^5)},
\end{split}
\]
where we used \eqref{t_0T_0} in the last inequality. By summing over $k$ and taking into account that $\varepsilon$ is sufficiently small the claim is proved.

Now using the embedding $W^{\frac{1}{2},\frac{30}{13}}_{x}(\R^5)\hookrightarrow L^{3}_{x}(\R^5)$ and the facts that $(6,\frac{30}{13})$ and $(3,\frac{30}{11})$ are admissible pairs, we deduce
\[
\begin{split}
	\|F_2\|_{L_t^{6}L_x^{3}([t_0,\infty)\times\R^5)}&\lesssim 	\|F_2\|_{L_t^{6}W_x^{\frac{1}{2},\frac{30}{13}}([t_0,\infty)\times\R^5)}\\
	&\lesssim \|f_k(\ub)\|_{L_t^{\frac{3}{2}}W_x^{\frac{1}{2},\frac{30}{19}}([t_0,\infty)\times\R^5)}\\
	&\lesssim \|\ub\|_{\Lb_t^{6}\Lb_x^{3}([t_0-\varepsilon^{-\alpha},t_0]\times\R^5)} 	\|\ub\|_{\Lb^2_t\Wb_x^{\frac{1}{2},\frac{10}{3}}([t_0-\varepsilon^{-\alpha},t_0]\times\R^5)}\\
	&\lesssim \varepsilon^\mu,
\end{split}
\]
where in the last inequality we used \eqref{t_0T_0} and the above claim.

Next we estimate $F_1$. First note that from Lemma \ref{dispest} and Lemma \ref{estdiffk}-(i) we obtain
\[
\begin{split}
	\|U_k(t-s)f_k(\ub)\|_{L^6_x}&\lesssim |t-s|^{-\frac{5}{3}}\|f_k(\ub)\|_{L_x^{\frac{6}{5}}}\\
	&\lesssim |t-s|^{-\frac{5}{3}}\|\ub\|^2_{\Lb_x^{\frac{12}{5}}}\\
	&\lesssim E_0^2|t-s|^{-\frac{5}{3}},
\end{split}
\]
where we used the embedding $H^1(\R^5)\hookrightarrow L^{\frac{12}{5}}(\R^5)$ and \eqref{gloestH1}. This last inequality gives
\begin{equation}\label{F111}
	\|F_1\|_{L_t^{3}L_x^{6}([t_0,\infty)\times\R^5)}\lesssim E_0^2\varepsilon^{\frac{\alpha}{3}}.
\end{equation}
Also, a simple computation yields
$$
i\int_0^{t_0-\varepsilon^{-\alpha}}U_k(t-s)f_k(\ub)ds=U_k(t-t_0+\varepsilon^{-\alpha})u_k(t_0-\varepsilon^{-\alpha})-U_k(t)u_{k0},
$$
from which we deduce 
\begin{equation}\label{F112}
\|F_1\|_{L_t^{\infty}L_x^{2}([t_0,\infty)\times\R^5)}\lesssim \|u_k(t_0-\varepsilon^{-\alpha})\|_{L^2_x}+\|u_{k0}\|_{L^2_x}\lesssim E_0.
\end{equation}
Interpolation and \eqref{F111}-\eqref{F112} imply
\[ 
\begin{split}
\|F_1\|_{L_t^{6}L_x^{3}([t_0,\infty)\times\R^5)}&\lesssim \|F_1\|^{\frac{1}{2}}_{L_t^{3}L_x^{6}([t_0,\infty)\times\R^5)}
\|F_1\|^{\frac{1}{2}}_{L_t^{\infty}L_x^{2}([t_0,\infty)\times\R^5)}
\lesssim  \varepsilon^{\frac{\alpha}{6}}.
\end{split}
\]
 The proof of the lemma is thus completed.
\end{proof}

\begin{proof}[Proof of Theorem \ref{scatteringcriteio}]
	From Lemma \ref{scat1} it suffices to prove that \eqref{sccrit} holds. To do this, observe first that the embedding $W^{\frac{1}{2},\frac{30}{13}}_{x}(\R^5)\hookrightarrow L^{3}_{x}(\R^5)$ and the Strichartz estimate give that $\|U_k(t)u_{k0}\|_{L_t^{6}L_x^{3}(\R\times\R^5)}\lesssim E_0$. Thus     we may decompose $\R=\bigcup_{j=1}^JI_j$ such that
	\begin{equation}\label{2.9}
		\|U_k(t)u_{k0}\|_{L_t^{6}L_x^{3}(I_j\times\R^5)}<\varepsilon^\nu, \qquad k=1,\ldots,l.
	\end{equation}
Also note that in order to obtain \eqref{sccrit} it suffices that
\begin{equation}\label{sccrit1}
		\|\ub\|_{\Lb_t^{6}\Lb_x^{3}(I_j\times\R^5)}<\infty, \qquad j=1,\ldots,J.
\end{equation}
Assume first that $|I_j|\leq 2T_0$. Then, from the embedding $H^{1}_{x}(\R^5)\hookrightarrow L^{3}_{x}(\R^5)$ we have
\begin{equation}\label{finiten}
	\|u_k\|^6_{L_t^{6}L_x^{3}(I_j\times\R^5)}\lesssim \|u_k\|^6_{L_t^{6}H_x^{1}(I_j\times\R^5)}\lesssim E_0^6|I_j|\lesssim E_0^6T_0<\infty, \quad k=1,\ldots,l.
\end{equation}
Hence, for \eqref{sccrit1} it suffices to consider those $I_j$ such that $|I_j|>2T_0$. So, fix such a $I_j$ and suppose that $I_j=(a_j,b_j)$ with (without loss of generality) $a_j\in\R$. From Lemma \ref{firstlemma} we may choose $t_0\in (a_j,a_j+T_0)$ such that \eqref{esL6L3} holds.

Now we write $I_j=(a_j,t_0]\cup(t_0,b_j)$. Since $t_0-a_j<T_0$ as in \eqref{finiten} we obtain that $	\|\ub\|_{\Lb_t^{6}\Lb_x^{3}((a_j,t_0]\times\R^5)}<\infty$. Thus it remains to prove that $	\|\ub\|_{\Lb_t^{6}\Lb_x^{3}((t_0,b_j)\times\R^5)}<\infty$. For this, note that
$$
U_k(t-t_0)u_k(t_0)=U_k(t)u_{k0}+i\int_0^{t_0}U_k(t-s)\frac{1}{\alpha_{k}}f_k(\ub)ds.
$$
Therefore,  \eqref{2.9} and Lemma \ref{firstlemma} imply
\begin{equation} \label{2.10}
	\|U_k(t-t_0)u_k(t_0)\|_{L_t^{6}L_x^{3}((t_0,b_j)\times\R^5)}\lesssim  \varepsilon^\nu.
\end{equation}
On the other hand, by writing
$$
u_k(t)=U_k(t-t_0)u_k(t_0)+i\int_{t_0}^tU_k(t-s)\frac{1}{\alpha_{k}}f_k(\ub)ds,
$$ 
Proposition \ref{stricha}, \eqref{2.10} and Lemma \ref{estdiffk}-(i) give
\[
\begin{split}
\|u_k\|_{L_t^{6}L_x^{3}((t_0,b_j)\times\R^5)}&\lesssim \varepsilon^\nu+\|f_k(\ub)\|_{L_t^{3}L_x^{\frac{3}{2}}((t_0,b_j)\times\R^5)}\\
&\lesssim  \varepsilon^\nu+\|\ub\|^2_{L_t^{6}L_x^{3}((t_0,b_j)\times\R^5)}.
\end{split}
\]
By summing over $k$, choosing $\varepsilon$ small enough and using a continuity argument we obtain
$$
\|\ub\|_{\Lb_t^{6}\Lb_x^{3}((t_0,b_j)\times\R^5)}\lesssim \varepsilon^\nu,
$$
which completes the proof of the theorem.
\end{proof}

 \section{Proof of the main result}\label{sec.proofscatt}
 This section is devoted to prove Theorem \ref{thm:critscatt}. The main ingredient is a Morawetz-type estimate. To prove it we first state some useful lemmas based on the ground state solutions and certain suitable radial functions. Throughout the section we assume the assumptions  of Theorem \ref{thm:critscatt}.

 \subsection{A Morawetz estimate}
 
 In this subsection we will prove a Morawetz estimate adapted to solutions of system \eqref{system1}.  Here,  $\chi$ denotes the radial cut-off function introduced in Section \ref{subseccoerxci}.

Following the strategy in \cite{dodsonmurphy2018} we introduce the some  radial functions. For  $R\gg 1$, set

\begin{equation}\label{defphi}
    \phi(x)=\frac{1}{\omega_{5}R^{5}}\int_{\R^{5}}\chi^{2}\left(\frac{x-s}{R}\right)\chi^{2}\left(\frac{s}{R}\right)\;ds, 
\end{equation}
where $\omega_{5}$ is the volume of the unit ball in $\R^{5}$ and

\begin{equation}\label{phi1def}
    \phi_{1}(x,y)=\frac{1}{\omega_{5}R^{5}}\int_{\R^{5}}\chi^{2}\left(\frac{y-s}{R}\right)\chi^{3}\left(\frac{x-s}{R}\right)\;ds.  
\end{equation}
Also, we set
\begin{equation}
    \varphi(x)=\frac{1}{|x|}\int_{0}^{|x|}\phi(r)\;dr. 
\end{equation}

We now collect some properties of the above functions.

\begin{lemma}\label{properrad}
  We have the following.
\begin{enumerate}
    \item[(i)] \begin{equation}\label{properrad1}
    |\varphi(x)|\lesssim\min\left\{1,\frac{R}{|x|}\right\}\qquad \mathrm{and} \qquad \partial_{j}\varphi(x)=\frac{x_{j}}{|x|^{2}}[\phi(x)-\varphi(x)],
\end{equation}
where $\partial_{j}$ stands for the derivative with respect to $x_j$.
 \item[(ii)] 
 \begin{equation}\label{properrad2}
    |\nabla \phi(x)|\lesssim \frac{1}{R}\qquad \mathrm{and} \qquad |\nabla\varphi(x)|=\left|\frac{x}{|x|^{2}}[\phi(x)-\varphi(x)]\right|\lesssim \min\left\{\frac{1}{R},\frac{R}{|x|^{2}}\right\}.
\end{equation}
    \item[(iii)] 
    \begin{equation}\label{properrad3}
       \varphi-\phi \geq 0, \quad |\varphi(x)-\phi(x)|\lesssim \min\left\{\frac{|x|}{R},\frac{R}{|x|}\right\}\quad \mathrm{and} \quad |\phi(x-y)-\phi_{1}(x,y)|\lesssim \varepsilon.
    \end{equation}
    \item[(iv)]
    \begin{equation}\label{properrad4}
        \partial_{j}[\varphi(x)x_{m}]=\phi(x)\delta_{jm}+(\varphi-\phi)(x)P_{jm}(x),
    \end{equation}
    where $\delta_{jm}$ denotes the  Kronecker delta and $P_{jm }(x)=\delta_{jm}-\frac{x_{j}x_{m}}{|x|^{2}}$. In particular, if $j=m$ we obtain
    \begin{equation}\label{properrad5}
        \sum_{j=1}^{5} \partial_{j}[\varphi(x)x_{j}]=5\phi(x)+4(\varphi-\phi)(x)=4\varphi(x)+\phi(x). 
    \end{equation}
\end{enumerate}  
\end{lemma}
\begin{proof}
These properties can be checked directly. However, for \eqref{properrad1} and \eqref{properrad2} see \cite{dodsonmurphy2018} pages 1814 and 1819, respectively (see also \cite[Lemma 4.2]{dinh2020scattering}). For \eqref{properrad3} see \cite[Lemma 4.2]{dinh2020scattering}. Identity \eqref{properrad4} follows directly from \eqref{properrad1}.
\end{proof}

In what follows, to simplify notation, we will introduce the  quantities
\begin{equation*}
    \mathcal{M}(\ub):=\sum_{k=1}^{l}\frac{\alpha_k^{2}}{\gamma_{k}}|u_{k}|^{2},\qquad\mathcal{K}(\ub):=\sum_{k=1}^{l}\gamma_{k}|\nabla u_{k}|^{2}\qquad\mbox{and}\qquad\mathcal{T}(\ub):=2\,\mathrm{Im}\sum_{k=1}^{l}\alpha_{k} \nabla u_{k} \overline{u}_{k},
\end{equation*}
and the functional
\begin{equation}\label{mdef}
    \M(\ub):=\sum_{k=1}^{l}\frac{\alpha_k^{2}}{\gamma_{k}}\|u_{k}\|^{2}_{L^{2}}=\int_{\R^{5}}\mathcal{M}(\ub)\;dx.
\end{equation}
Note in particular that function $K$ in \eqref{funclKP6} writes as
   \begin{equation*}
 K(\ub)=\int_{\R^{5}}\mathcal{K}(\ub)\;dx.
   \end{equation*}

Finally, we introduce the interaction Morawetz quantity associated to a global solution $\ub$ of \eqref{system1} as
 \begin{equation*}
	M_{R}(t)=\int_{\R^{5}\times \R^{5}}\varphi(x-y)(x-y)\cdot\mathcal{T}(\ub)(x,t)\mathcal{M}(\ub)(y,t)\;dx\;dy.
\end{equation*}
Note that by using  Cauchy-Schwarz inequality, \eqref{properrad1}  and the fact that $\ub$ is uniformly bounded in $\mathbf{H}_x^{1}$ we have
\begin{equation}\label{boundMR}
	\sup_{t\in \R}|M_{R}(t)|\lesssim R. 
\end{equation}

To establish the Morawetz estimate we will need the following results.

 \begin{lemma}\label{Virialinde} Assume $ \mathbf{u}_0 \in \mathbf{H}_{x}^{1}$ and  let $\mathbf{u}$ be the corresponding solution of \eqref{system1}, then

\begin{enumerate}
    \item[(i)] 
    \begin{equation}\label{derivmassre}
\begin{split}
\frac{d}{dt}\mathcal{M}(\ub)&=-2\sum_{j=1}^{5}\sum_{k=1}^{l}\;\partial_{j}\mathrm{Im}\left[ \alpha_{k}\partial_{j} u_{k} \overline{u}_{k}\right]-2\mathrm{Im}\sum_{k=1}^{l}\frac{\alpha_{k}}{\gamma_{k}}f_{k}(\ub)\overline{u}_{k}.
\end{split}
\end{equation}
    \item[(ii)]
     \begin{equation}\label{derivmoment}
\begin{split}
\frac{d}{dt}\mathcal{T}(\ub)&=-4\sum_{m=1}^{5}\mathrm{Re}\;\partial_{m}\left[\sum_{k=1}^{l}\gamma_{k}\nabla\overline{u}_{k}\partial_{m}u_{k}\right]+\nabla\Delta\left(\sum_{k=1}^{l}\gamma_{k}|u_{k}|^{2}\right)+2\nabla\mathrm{Re}\, F\left(\mathbf{u}\right).
\end{split}
\end{equation}
\end{enumerate}

\end{lemma}
\begin{proof}
These properties are consequence of the fact that $\mathbf{u}$ is a solution of \eqref{system1}. For (ii) it is necessary to use Lemma \ref{estdiffk}-(ii).  
\end{proof}

\begin{lemma}\label{Gaugeprop}
For $\xi \in \R^{5}$, let $\ub^\xi$ be defined as in \eqref{Gaugetrans}.
Then
\begin{enumerate}
\item[(i)]
\begin{equation*}
    \mathcal{M}(\ub^{\xi}(x))= \mathcal{M}(\ub(x)).
\end{equation*}
    \item[(ii)]
    \begin{equation*}
        \left|\nabla[u^{\xi}_{k}(x)]\right|^{2}=\frac{\alpha_{k}^{2}}{\gamma_{k}^{2}}|\xi|^{2}| u_{k}(x)|^{2}+2\frac{\alpha_{k}}{\gamma_{k}}\xi\cdot\mathrm{Im}[\nabla u_{k}\overline{u}_{k}](x)+\left|\nabla u_{k}(x)\right|^{2}
    \end{equation*}
    In particular,
    \begin{equation*}
        \mathcal{K}(\ub^{\xi}(x))=|\xi|^{2}\mathcal{M}(\ub(x))+\xi\cdot\mathcal{T}(\ub(x))+\mathcal{K}(\ub(x)).
    \end{equation*}
    \item[(iii)]
    \begin{equation*}
         [\nabla u_{k}^{\xi}\,\overline{u}_{k}^{\xi}](x)=\frac{\alpha_{k}}{\gamma_{k}}i\xi|u_{k}(x)|^{2}+[\nabla u_{k}\,\overline{u}_{k}](x)
    \end{equation*}
    In particular,
    \begin{equation}\label{GaugeT}
        \frac{1}{2}\mathcal{T}(\ub^{\xi}(x))=\xi\mathcal{M}(\ub(x))+\frac{1}{2}\mathcal{T}(\ub(x)).
    \end{equation}
\end{enumerate}
\begin{proof}
The proof follows by simple computations. So we omit the details.
\end{proof}

\end{lemma}

\begin{lemma}\label{angulderiv}
 Let $f$ and $g$  be  smooth functions defined on $\R^{5}$. Let $\nnabla\;_{z}$   denote the angular derivative\\ centered at $z$, that is,
\begin{equation}
   \nnabla\;_{z}f(x) =\nabla f(x)-\frac{x-z}{|x-z|}\left(\frac{x-z}{|x-z|}\cdot\nabla f(x)\right).  
 \end{equation}
 Then, with the same notation of Lemma \ref{properrad} we have
 \begin{itemize}
     \item[(i)] 
     \begin{equation*}
\sum_{j=1}^{5}\sum_{m=1}^{5}P_{jm}(x-y)\mathrm{Im}[\overline{f}\,\partial_{m}f](x)\cdot \mathrm{Im}[\overline{g}\,\partial_{j}g](y)= \mathrm{Im}[\overline{f}\,\nnabla\;_{y}f](x)\cdot \mathrm{Im}[\overline{g}\,\nnabla\;_{x}g](y).
 \end{equation*}
     \item[(ii)]
     \begin{equation*}
          \sum_{j=1}^{5}\sum_{m=1}^{5}\mathrm{Re}[\partial_{j}\overline{f}\partial_{m}f](x)P_{jm}(x-y)=|\nnabla\;_{y}f](x)|^{2}.
     \end{equation*}
 \end{itemize}
 \end{lemma}
\begin{proof}
Both properties follow from the definition of the angular derivative. In fact,
\begin{equation*}
    \begin{split}
         \mathrm{Im}[\overline{f}\,\nnabla\;_{y}f](x)\cdot \mathrm{Im}[\overline{g}\,\nnabla\;_{x}g](y)&=\sum_{j=1}^{5}\left\{\mathrm{Im}[\overline{f}\,\partial_{j}f](x)-\frac{(x-y)_{j}}{|x-y|}\sum_{m=1}^{5}\frac{(x-y)_{m}}{|x-y|}\mathrm{Im}[\overline{f}\,\partial_{m}f](x)\right\}\\
         &\qquad\times\left\{\mathrm{Im}[\overline{g}\,\partial_{j}g](y)-\frac{(x-y)_{j}}{|x-y|}\sum_{m=1}^{5}\frac{(x-y)_{m}}{|x-y|}\mathrm{Im}[\overline{g}\,\partial_{m}g](y)\right\}\\
         &=\sum_{j=1}^{5}\mathrm{Im}[\overline{f}\,\partial_{j}f](x)\mathrm{Im}[\overline{g}\,\partial_{j}g](y)\\
         &\quad-\sum_{j=1}^{5}\sum_{m=1}^{5}\frac{(x-y)_{j}}{|x-y|}\frac{(x-y)_{m}}{|x-y|}\mathrm{Im}[\overline{f}\,\partial_{j}f](x)\mathrm{Im}[\overline{g}\,\partial_{m}g](y).
      \end{split}
  \end{equation*}  
 Now, rearranging the terms we infer 
  \begin{equation*}
  	\begin{split}       
    \mathrm{Im}[\overline{f}\,\nnabla\;_{y}f](x)\cdot \mathrm{Im}[\overline{g}\,\nnabla\;_{x}g](y)      &=\sum_{j=1}^{5}\sum_{m=1}^{5}\mathrm{Im}[\overline{f}\,\partial_{j}f](x)\mathrm{Im}[\overline{g}\,\partial_{m}g](y)\delta_{jm}\\
         &\quad-\sum_{j=1}^{5}\sum_{m=1}^{5}\frac{(x-y)_{j}}{|x-y|}\frac{(x-y)_{m}}{|x-y|}\mathrm{Im}[\overline{f}\,\partial_{j}f](x)\mathrm{Im}[\overline{g}\,\partial_{m}g](y)\\
         &=\sum_{j=1}^{5}\sum_{m=1}^{5}\mathrm{Im}[\overline{f}\,\partial_{j}f](x)\mathrm{Im}[\overline{g}\,\partial_{m}g](y)P_{jm}(x-y).
    \end{split}
\end{equation*}
This proves (i). For (ii) a direct calculation gives
\begin{equation*}
    \begin{split}
        \sum_{j=1}^{5}\sum_{m=1}^{5}\mathrm{Re}[\partial_{j}\overline{f}\partial_{m}f](x)P_{jm}(x-y)&=\mathrm{Re}\left[|\nabla f(x)|^{2}-\overline{\left(\sum_{j=1}^{5}\frac{(x-y)_{j}}{|x-y|}\partial_{j}f(x)\right)}\left(\sum_{m=1}^{5}\frac{(x-y)_{m}}{|x-y|}\partial_{m}f(x)\right)\right]\\
        &=|\nabla f(x)|^{2}-\left|\frac{(x-y)}{|x-y|}\cdot\nabla f(x)\right|^2\\
        &=|\nnabla\;_{y}f(x)|^{2},
    \end{split}
\end{equation*}
which is the desired statement.
\end{proof}

The next result is our Morawetz estimate. We follow the ideas presented in \cite{dodsonmurphy2018} and \cite{wang2019Sacttering}.

\begin{proposition}[Interaction Morawetz estimate]\label{viri-Morz} Let $\ub_0\in \mathbf{H}^{1}_{x}$  satisfy \eqref{desEQgs}-\eqref{desQKgs} and let $\mathbf{u}$ be the corresponding  solution of \eqref{system1} given by Theorem A.
Let $\varepsilon>0$ and $a\in \R$. Then there are $\delta>0$,  $T_{0}=T_0(\varepsilon)\geq1$, $J=J(\varepsilon)$ and $\varepsilon_{1}>0$ sufficiently small such that if
\begin{equation}\label{1.17.1}
    \left|\frac{\alpha_{k}}{\gamma_{k}}-\sigma_{k}\right|<\varepsilon_{1}, \qquad k=1,\ldots,l,
\end{equation} 
then for $R_{0}=R_{0}(\delta,Q(\ub_0),\psib)$ sufficiently large we have
\begin{equation}\label{virmorineq}
    \frac{\delta}{JT_{0}}\int_{a}^{a+T_{0}}\int_{R_{0}}^{R_{0}e^{J}}\frac{1}{R^{5}}\int_{\R^{5}}\left[\M\left(\chi\left(\frac{\cdot-s}{R}\right)\ub\right)K\left(\chi\left(\frac{\cdot-s}{R}\right)\ub^{\xi_{0}}\right)\right]ds\frac{dR}{R}dt\lesssim\varepsilon,
\end{equation}
where $\chi$ is as in \eqref{defchi}  and
\begin{equation}\label{xi_0def}
\xi_{0}=\left\{\begin{array}{cl}
	{\displaystyle-\frac{1}{2}\frac{\displaystyle\int_{\R^{5}}\chi^{2}\left(\frac{x-s}{R}\right)\mathcal{T}(\ub)(x)\;dx}{\displaystyle\int_{\R^{5}}\chi^{2}\left(\frac{x-s}{R}\right)\mathcal{M}(\ub)(x)\;dx}, }& \quad \mbox{if}\quad\displaystyle \int_{\R^{5}}\chi^{2}\left(\frac{x-s}{R}\right)\mathcal{M}(\ub)(x)\;dx\neq 0,\\\\
	0,& \quad\mbox{otherwise}.
\end{array}\right.
\end{equation}
\end{proposition}
\begin{proof}
To shorten notation, we will drop the dependence of $t$ in what follows. We start using identity  \eqref{derivmoment} to get
\begin{equation}
\begin{split}\label{derMR}
   \frac{dM_{R}(t)}{dt}&=2\int_{\R^{5}\times \R^{5}}\varphi(x-y)(x-y)\cdot\nabla\mathrm{Re}\, F\left(\mathbf{u}(x)\right)\mathcal{M}(\ub(y))\;dx\;dy\\
        &\quad+\int_{\R^{5}\times \R^{5}}\varphi(x-y)(x-y)\cdot\nabla\Delta\left(\sum_{k=1}^{l}\gamma_{k}|u_{k}(x)|^{2}\right)\mathcal{M}(\ub(y))\;dx\;dy\\
        &\quad-4\sum_{j=1}^{5}\int_{\R^{5}\times \R^{5}}\varphi(x-y)(x-y)_{j}\sum_{m=1}^{5}\mathrm{Re}\;\partial_{m}\left[\sum_{k=1}^{l}\gamma_{k}\partial_{j}\overline{u}_{k}\partial_{m}u_{k}\right](x)\\
        &\quad\qquad\qquad\qquad\qquad\qquad\qquad\qquad\qquad\qquad\qquad\times\mathcal{M}(\ub(y))\;dx\;dy\\
        &\quad+\int_{\R^{5}\times \R^{5}}\varphi(x-y)(x-y)\cdot\mathcal{T}(\ub)(x)\frac{d}{dt}\mathcal{M}(\ub(y))\;dx\;dy\\
        &=:\mathcal{I}_{1}+\mathcal{I}_{2}+\mathcal{I}_{3}+\mathcal{I}_{4}.
        \end{split}
\end{equation}

We first work with $\mathcal{I}_{1}$. Integrating by parts and using \eqref{properrad5}  we obtain
\begin{equation*}
\begin{split}
    \mathcal{I}_{1}&=-2\sum_{j=1}^{5}\int_{\R^{5}\times \R^{5}}\partial_{j}[\varphi(x-y)(x-y)_{j}]\mathrm{Re}\, F\left(\mathbf{u}(x)\right)\mathcal{M}(\ub(y))\;dx\;dy\\
  &=-2\sum_{j=1}^{5}\int_{\R^{5}\times \R^{5}}[5\phi(x-y)+4(\varphi-\phi)(x-y)]\mathrm{Re}\, F\left(\mathbf{u}(x)\right)\mathcal{M}(\ub(y))\;dx\;dy.
\end{split}
\end{equation*}

Then, using  \eqref{phi1def}, hypotheses \ref{H5} and the definition of $\M$ (see \eqref{mdef})  we can rewrite the last expression as
\begin{equation}
    \begin{split}
        \mathcal{I}_{1}
  &=\frac{-10}{\omega_{5}R^{5}}\int_{\R^{5}}\left[\M\left(\chi\left(\frac{\cdot-s}{R}\right)\ub\right)\int_{\R^{5}}\mathrm{Re}\, F\left(\chi\left(\frac{x-s}{R}\right)\mathbf{u}(x)\right)\;dx\right]\;ds\label{term1.1}\\
  &\quad-8\int_{\R^{5}\times \R^{5}}\mathrm{Re}\, F\left(\mathbf{u}(x)\right)\mathcal{M}(\ub(y))[(\varphi-\phi)(x-y)]\;dx\;dy\\
  &\quad -10\int_{\R^{5}\times \R^{5}}\mathrm{Re}\, F\left(\mathbf{u}(x)\right)\mathcal{M}(\ub(y))[\phi(x-y)-\phi_{1}(x,y)]\;dx\;dy.
    \end{split}
\end{equation}

Next, we consider $\mathcal{I}_{2}$. Integrating by parts twice (with respect to $x$) and using \eqref{properrad5} we obtain
\begin{equation}
  \mathcal{I}_{2}=  \int_{\R^{5}\times \R^{5}}\nabla[4\varphi(x-y)+\phi(x-y)]\cdot\nabla\left(\sum_{k=1}^{l}\gamma_{k}|u_{k}(x)|^{2}\right)\mathcal{M}(\ub(y))\;dx\;dy.\label{term2.1}
\end{equation}
Now, we work with $\mathcal{I}_{3}$. Integrating by parts and using the definition of $P_{jm}$ (see Lemma \ref{properrad}-(iv)) we have
\begin{align*}
              \begin{split}
              \mathcal{I}_{3} &=4\sum_{j=1}^{5}\sum_{m=1}^{5}\int_{\R^{5}\times \R^{5}}\partial_{m}[\varphi(x-y)(x-y)_{j}]\mathrm{Re}\;\left[\sum_{k=1}^{l}\gamma_{k}\partial_{j}\overline{u}_{k}\partial_{m}u_{k}\right](x)\\
        &\quad\qquad\qquad\qquad\qquad\qquad\qquad\qquad\qquad\qquad\qquad\times\mathcal{M}(\ub(y))\;dx\;dy\\
        &=4\sum_{j=1}^{5}\sum_{m=1}^{5}\int_{\R^{5}\times \R^{5}}\delta_{mj}\phi(x-y)\mathrm{Re}\;\left[\sum_{k=1}^{l}\gamma_{k}\partial_{j}\overline{u}_{k}\partial_{m}u_{k}\right](x)\mathcal{M}(\ub(y))\;dx\;dy\\
        &\quad+4\sum_{j=1}^{5}\sum_{m=1}^{5}\int_{\R^{5}\times \R^{5}}P_{mj}(x-y)(\varphi-\phi)(x-y)\mathrm{Re}\;\left[\sum_{k=1}^{l}\gamma_{k}\partial_{j}\overline{u}_{k}\partial_{m}u_{k}\right](x)\\
        &\quad\qquad\qquad\qquad\qquad\qquad\qquad\qquad\qquad\qquad\qquad\times\mathcal{M}(\ub(y))\;dx\;dy.
  \end{split}
\end{align*}
Using now the definition of $\phi$ and the fact that $\chi$ is radial we may write
\begin{equation}
	\begin{split}\label{I3}
     \mathcal{I}_{3}  &=4\int_{\R^{5}\times \R^{5}}\phi(x-y)\mathrm{Re}\;\left[\sum_{k=1}^{l}\gamma_{k}|\nabla u_{k}|^{2}\right](x)\mathcal{M}(\ub(y))\;dx\;dy\\
         &\quad+4\sum_{j=1}^{5}\sum_{m=1}^{5}\int_{\R^{5}\times \R^{5}}P_{mj}(x-y)(\varphi-\phi)(x-y)\mathrm{Re}\;\left[\sum_{k=1}^{l}\gamma_{k}\partial_{j}\overline{u}_{k}\partial_{m}u_{k}\right](x)\\
        &\quad\qquad\qquad\qquad\qquad\qquad\qquad\qquad\qquad\qquad\qquad\times\mathcal{M}(\ub(y))\;dx\;dy\\
           &=\frac{4}{\omega_{5}R^{5}}\int_{\R^{5}\times \R^{5}}\int_{\R^{5}}\chi^{2}\left(\frac{x-s}{R}\right)\chi^{2}\left(\frac{y-s}{R}\right)\mathcal{K}(\ub(x))\mathcal{M}(\ub(y))\;dx\;dy\;ds\\
            &\quad+4\sum_{j=1}^{5}\sum_{m=1}^{5}\int_{\R^{5}\times \R^{5}}P_{mj}(x-y)(\varphi-\phi)(x-y)\mathrm{Re}\;\left[\sum_{k=1}^{l}\gamma_{k}\partial_{j}\overline{u}_{k}\partial_{m}u_{k}\right](x)\\
        &\quad\qquad\qquad\qquad\qquad\qquad\qquad\qquad\qquad\qquad\qquad\times\mathcal{M}(\ub(y))\;dx\;dy \\
        &=:\mathcal{I}_{31}+\mathcal{I}_{32}.
           \end{split}
\end{equation}
Finally, we consider the term $\mathcal{I}_{4}$. We use \eqref{derivmassre} to write, 
\begin{equation}
     \mathcal{I}_{4}=\mathcal{I}-2\int_{\R^{5}\times \R^{5}}\varphi(x-y)(x-y)\cdot\mathcal{T}(\ub)(x)\left(\mathrm{Im}\sum_{k=1}^{l}\frac{\alpha_{k}}{\gamma_{k}}f_{k}(\ub)\overline{u}_{k}\right)(y)\;dx\;dy,\label{term4.1}
\end{equation}
where, integration by parts and \eqref{properrad4}, yield
      \begin{align*}
              \begin{split}
           \mathcal{I} &=-4\int_{\R^{5}\times \R^{5}}\varphi(x-y)(x-y)\cdot\left(\sum_{k=1}^{l}\alpha_{k}\mathrm{Im} \nabla u_{k} \overline{u}_{k}\right)(x)\\
      &\qquad\qquad\qquad\qquad\qquad\qquad\qquad\qquad\times\sum_{j=1}^{5}\sum_{k=1}^{l}\;\partial_{j}\mathrm{Im}\left[ \alpha_{k}\partial_{j} u_{k} \overline{u}_{k}\right](y)\;dx\;dy\\
        &=-4\sum_{j=1}^{5}\sum_{m=1}^{5}\int_{\R^{5}\times \R^{5}}\delta_{jm}\phi(x-y)\sum_{k=1}^{l}\alpha_{k}\mathrm{Im} [\partial_{m} u_{k} \overline{u}_{k}](x)\\
        &\qquad\qquad\qquad\qquad\qquad\qquad\qquad\qquad\qquad\qquad\times\sum_{k=1}^{l}\alpha_{k}\mathrm{Im}\left[ \partial_{j} u_{k} \overline{u}_{k}\right](y)\;dx\;dy\\
        &\quad-4\sum_{j=1}^{5}\sum_{m=1}^{5}\int_{\R^{5}\times \R^{5}}P_{jm}(x-y)(\varphi-\phi)(x-y)\sum_{k=1}^{l}\alpha_{k}\mathrm{Im} [\partial_{m} u_{k} \overline{u}_{k}](x)\\
        &\qquad\qquad\qquad\qquad\qquad\qquad\qquad\qquad\qquad\qquad\times\sum_{k=1}^{l}\alpha_{k}\mathrm{Im}\left[ \partial_{j} u_{k} \overline{u}_{k}\right](y)\;dx\;dy.
         \end{split}
    \end{align*}
   Using the definition of $\mathcal{T}$ we then see that     
     \begin{align}
     	\begin{split}   
        \mathcal{I} &=-\int_{\R^{5}\times \R^{5}}\phi(x-y)\mathcal{T}(\ub(x))\cdot\mathcal{T}(\ub(y))\;dx\;dy\nonumber\\
          &\quad-4\sum_{j=1}^{5}\sum_{m=1}^{5}\int_{\R^{5}\times \R^{5}}P_{jm}(x-y)(\varphi-\phi)(x-y)\sum_{k=1}^{l}\alpha_{k}\mathrm{Im} [\partial_{m} u_{k} \overline{u}_{k}](x)\\
        &\qquad\qquad\qquad\qquad\qquad\qquad\qquad\qquad\qquad\qquad\times\sum_{k=1}^{i}\alpha_{k}\mathrm{Im}\left[ \partial_{j} u_{k} \overline{u}_{k}\right](y)\;dx\;dy.
           \end{split}
\end{align}

From  \eqref{defphi} and  the last equality, \eqref{term4.1} reads as
\begin{equation}
   \begin{split}\label{I4}
           \mathcal{I}_{4}&=\frac{-1}{\omega_{5}R^{5}}\int_{\R^{5}\times \R^{5}}\int_{\R^{5}}\chi^{2}\left(\frac{x-s}{R}\right)\chi^{2}\left(\frac{y-s}{R}\right)\mathcal{T}(\ub(x))\cdot\mathcal{T}(\ub(y))\;dx\;dy\;ds\\
              &\quad-4\sum_{j=1}^{5}\sum_{m=1}^{5}\int_{\R^{5}\times \R^{5}}P_{jm}(x-y)(\varphi-\phi)(x-y)\sum_{k=1}^{l}\alpha_{k}\mathrm{Im} [\partial_{m} u_{k} \overline{u}_{k}](x)\\
        &\qquad\qquad\qquad\qquad\qquad\qquad\qquad\qquad\qquad\qquad\times\sum_{k=1}^{1}\alpha_{k}\mathrm{Im}\left[ \partial_{j} u_{k} \overline{u}_{k}\right](y)\;dx\;dy\\
           &\quad-2\int_{\R^{5}\times \R^{5}}\varphi(x-y)(x-y)\cdot\mathcal{T}(\ub((x))\left(\mathrm{Im}\sum_{k=1}^{l}\frac{\alpha_{k}}{\gamma_{k}}f_{k}(\ub)\overline{u}_{k}\right)(y)\;dx\;dy\\
           &=:\mathcal{I}_{41}+\mathcal{I}_{42}+\mathcal{I}_{43}.
           \end{split}
\end{equation}

\noindent {\bf Claim 1.}
We have
\begin{equation*}
   \mathcal{I}_{32}+\mathcal{I}_{42}\geq 0.
\end{equation*}

First note that from Lemma \ref{angulderiv}-(ii) and the fact that $\varphi-\phi$ is a radial function we can write
\begin{equation*}
    \begin{split}
       \mathcal{I}_{32}&=4\int_{\R^{5}\times \R^{5}}(\varphi-\phi)(x-y)\left(\sum_{k=1}^{l}\gamma_{k}|\nnabla\;_{y} u_{k}(x)|^{2}\right)\mathcal{M}(\ub(y))\;dx\;dy\\
        &=2\int_{\R^{5}\times \R^{5}}(\varphi-\phi)(x-y)\left(\sum_{k=1}^{l}\gamma_{k}|\nnabla\;_{y} u_{k}(x)|^{2}\right)\mathcal{M}(\ub(y))\;dx\;dy\\
        &\quad+2\int_{\R^{5}\times \R^{5}}(\varphi-\phi)(x-y)\left(\sum_{k=1}^{l}\gamma_{k}|\nnabla\;_{x} u_{k}(y)|^{2}\right)\mathcal{M}(\ub(x))\;dx\;dy.
    \end{split}
\end{equation*}

On the other hand, using Lemma \ref{angulderiv}-(i) we obtain
\begin{equation*}
    \begin{split}
       \mathcal{I}_{42}&=-4\int_{\R^{5}\times \R^{5}}(\varphi-\phi)(x-y)\left(\sum_{k=1}^{l}\alpha_{k}\mathrm{Im} [\nnabla\;_{y} u_{k} \overline{u}_{k}](x)\right)\cdot\left(\sum_{k=1}^{l}\alpha_{k}\mathrm{Im}\left[ \nnabla\;_{x} u_{k} \overline{u}_{k}\right](y)\right)\;dx\;dy.
    \end{split}
\end{equation*}
Thus, 
\begin{equation*}
    \begin{split}
     \mathcal{I}_{32}+ \mathcal{I}_{42}&=2\left.\int_{\R^{5}\times \R^{5}}(\varphi-\phi)(x-y)\right[\\
       &\qquad\left(\sum_{k=1}^{l}\gamma_{k}|\nnabla\;_{y} u_{k}(x)|^{2}\right)\mathcal{M}(\ub(y))+\left(\sum_{k=1}^{l}\gamma_{k}|\nnabla\;_{x} u_{k}(y)|^{2}\right)\mathcal{M}(\ub(x))\\
        &\qquad-2\left.\left(\sum_{k=1}^{l}\alpha_{k}\mathrm{Im} [\nnabla\;_{y} u_{k} \overline{u}_{k}](x)\right)\cdot\left(\sum_{k=1}^{l}\alpha_{k}\mathrm{Im}\left[ \nnabla\;_{x} u_{k} \overline{u}_{k}\right](y)\right)\right] \;dx\;dy\\
        &=:2\int_{\R^{5}\times \R^{5}}(\varphi-\phi)(x-y)\mathcal{S}\;dx\;dy.
    \end{split}
\end{equation*}
By using the definition of $\mathcal{M}$, an application of the Cauchy-Schwarz inequality gives
\begin{equation*}
    \begin{split}
   \mathcal{S}\geq \sum_{j=1}^{l}\sum_{m=1}^{l}\left(\sqrt{\frac{\gamma_{m}}{\gamma_{j}}}\alpha_{j}|\nnabla\;_{y}u_{m}(x)||u_{j}(y)|-\sqrt{\frac{\gamma_{j}}{\gamma_{m}}}\alpha_{m}|\nnabla\;_{x}u_{j}(y)||u_{m}(x)|\right)^{2} \geq 0.
    \end{split}
\end{equation*}
From this and Lemma \ref{properrad}-(iii), the claim follows.\\

\noindent {\bf Claim 2.}
We have
\begin{equation*}
\begin{split}
   \mathcal{I}_{31}+\mathcal{I}_{41}
   &=\frac{4}{\omega_{5}R^{5}}\int_{\R^{5}}\left[\M\left(\chi\left(\frac{\cdot-s}{R}\right)\ub\right)\int_{\R^{5}}\chi^{2}\left(\frac{x-s}{R}\right)\mathcal{K}\left(\ub^{\xi_{0}}(x)\right)dx\right]\;ds,
\end{split}
\end{equation*}
where $\xi_0$ is given in \eqref{xi_0def} 

Indeed, for $s\in \R^{5}$ fixed,  we first note that the quantity
\begin{equation}\label{inv1}
    \int_{\R^{5}\times \R^{5}}\chi^{2}\left(\frac{x-s}{R}\right)\chi^{2}\left(\frac{y-s}{R}\right)\left[\mathcal{K}(\ub(x))\mathcal{M}(\ub(y))-\frac{1}{4}\mathcal{T}(\ub(x))\cdot\mathcal{T}(\ub(y))\right]\;dx\;dy
\end{equation}
is Galilean invariant, that is, it is invariant under the Gauge transformation \eqref{Gaugetrans}.  To see this, it follows from Lemma \ref{Gaugeprop} that
\begin{equation*}
\begin{split}
    \mathcal{K}(\ub^{\xi}(x))\mathcal{M}(\ub^{\xi}(y))-\frac{1}{4}\mathcal{T}(\ub^{\xi}(x))\cdot\mathcal{T}(\ub^{\xi}(y))&= \mathcal{K}(\ub(x))\mathcal{M}(\ub(y))-\frac{1}{4}\mathcal{T}(\ub(x))\cdot\mathcal{T}(\ub(y))\\
    &\quad+\frac{1}{2}\xi\cdot\mathcal{T}(\ub(x))\mathcal{M}(\ub(y))-\frac{1}{2}\xi\cdot\mathcal{T}(\ub(y))\mathcal{M}(\ub(x)). 
\end{split}
\end{equation*}
Hence, since $\chi^{2}$ is symmetric, inserting the last two terms of the above identity into \eqref{inv1}, a change of variables gives that both integrals are the same, which implies our assertion.  

Now we turn attention to $\mathcal{I}_{31}+\mathcal{I}_{41}$. The first step is to look  for some  $\xi_{0} \in \R^{5}$ such that
\begin{equation*}
    \int_{\R^{5}}\chi^{2}\left(\frac{x-s}{R}\right)\mathcal{T}(\ub^{\xi_0}(x))\;dx=0.
\end{equation*}
From \eqref{GaugeT} this can be achieved by taking $\xi_0$ as in \eqref{xi_0def}.
For this choice of $\xi_{0}$ we obtain
\begin{equation*}
\begin{split}
  \mathcal{I}_{31}+\mathcal{I}_{41}&= \frac{4}{\omega_{5}R^{5}}\int_{\R^{5}}\left\{\left[\int_{ \R^{5}}\chi^{2}\left(\frac{x-s}{R}\right)\mathcal{K}(\ub^{\xi_{0}}(x))\;dx\right]\left[\int_{ \R^{5}}\chi^{2}\left(\frac{y-s}{R}\right)\mathcal{M}(\ub(y))\;dy\right]\right\}\;ds,
\end{split}
\end{equation*} 
and the claim follows from the definition of $\M$. 

Gathering together \eqref{derMR}-\eqref{I3}, \eqref{I4}, Claim 1 and Claim 2, we have
\begin{equation}
\begin{split}\label{a}
    &\frac{1}{\omega_{5}R^{5}}\int_{\R^{5}}\left\{\left[\int_{ \R^{5}}4\chi^{2}\left(\frac{x-s}{R}\right)\mathcal{K}(\ub^{\xi_{0}}(x))\;dx\right.\right.\\
    &\qquad\qquad\qquad\qquad\left.\left.-10\int_{\R^{5}}\mathrm{Re}\, F\left(\chi\left(\frac{x-s}{R}\right)\mathbf{u}(x)\right)\;dx\right]\M\left(\chi\left(\frac{\cdot-s}{R}\right)\ub\right)\right\}\;ds\\
        &\leq\frac{dM_{R}}{dt}\\
        &+\int_{\R^{5}\times \R^{5}}|\mathrm{Re}\, F\left(\mathbf{u}(x)\right)|\mathcal{M}(\ub(y))|8(\varphi-\phi)(x-y)+10[\phi(x-y)-\phi_{1}(x,y)]|\;dx\;dy\\
        &+\int_{\R^{5}\times \R^{5}}|\nabla[4\varphi(x-y)+\phi(x-y)]|\left(2\sum_{k=1}^{l}\gamma_{k}| u_{k}(x)||\nabla u_{k}(x)|\right)\mathcal{M}(\ub(y))\;dx\;dy\\
        &+2\int_{\R^{5}\times \R^{5}}|\varphi(x-y)(x-y)\cdot\mathcal{T}(\ub)(x)|\left|\left(\mathrm{Im}\sum_{k=1}^{l}\frac{\alpha_{k}}{\gamma_{k}}f_{k}(\ub)\overline{u}_{k}\right)(y)\right|\;dx\;dy\\
        &=:\frac{dM_{R}}{dt}+\mathcal{A}+\mathcal{B}+\mathcal{C}.
        \end{split}
\end{equation}
 
 The idea now is to average this inequality over $t\in [a,a+T_{0}]$ and logarithmically over $R\in [R_{0},R_{0}e^{J}]$.

For $\displaystyle\frac{dM_{R}}{dt}$, from the fundamental theorem of calculus and \eqref{boundMR}  we have
\begin{equation}\label{int1}
    \left|\frac{1}{T_{0}J}\int_{a}^{a+T_{0}}\int_{R_{0}}^{R_{0}e^{J}}\frac{dM_{R}}{dt}\;\frac{dR}{R}\;dt\right|\lesssim \frac{1}{T_{0}J}R_{0}e^{J}.
\end{equation}

Next, recalling  \eqref{properrad3}, we have
$$
\int_{R_{0}}^{R_0e^J}|(\varphi-\phi)(x-y)|\frac{dR}{R}\lesssim \int_{R_{0}}^{R_0e^J}\min\left\{\frac{|x-y|}{R},\frac{R}{|x-y|}\right\}\frac{dR}{R}\lesssim 1.
$$
Thus, in view of Lemma \ref{estdiffk}-(i) we deduce
\begin{equation}\label{int1.1}
\begin{split}
 \bigg|\frac{1}{T_{0}J}&\int_{a}^{a+T_{0}}\int_{R_{0}}^{R_{0}e^{J}} \int_{\R^{5}\times \R^{5}}|\mathrm{Re}\, F\left(\mathbf{u}(x)\right)|\mathcal{M}(\ub(y))|8(\varphi-\phi)(x-y)| \;\frac{dR}{R}\;dt\bigg|\\
 &\lesssim \frac{1}{T_{0}J}\int_{a}^{a+T_{0}}\left(\int_{ \R^{5}}\sum_{k=1}^l|u_k(x)|^3dx\right)\left(\int_{\R^5}\mathcal{M}(u(y))dy\right)dt\\
& \lesssim \frac{1}{J},
\end{split}
\end{equation}
where we used the embedding $H^1(\R^5)\hookrightarrow L^3(\R^5)$ and the fact that $\ub(t)$ is uniformly bounded with respect to $t$.
In addition, from \eqref{properrad3},
$$
\int_{R_{0}}^{R_0e^J}|\varphi(x-y)-\phi_1(x,y)|\frac{dR}{R}\lesssim \varepsilon \int_{R_{0}}^{R_0e^J}\frac{dR}{R}\lesssim \varepsilon J.
$$
Therefore, as in \eqref{int1.1}, we now deduce
\begin{equation}\label{int1.2}
		\bigg|\frac{1}{T_{0}J}\int_{a}^{a+T_{0}}\int_{R_{0}}^{R_{0}e^{J}} \int_{\R^{5}\times \R^{5}}|\mathrm{Re}\, F\left(\mathbf{u}(x)\right)|\mathcal{M}(\ub(y))|10[\phi(x-y)-\phi_1(x,y)]| \;\frac{dR}{R}\;dt\bigg|\\
		\lesssim \varepsilon.
\end{equation}
A combination of \eqref{int1.1} and \eqref{int1.2} yields
\begin{equation}
    \left|\frac{1}{T_{0}J}\int_{a}^{a+T_{0}}\int_{R_{0}}^{R_{0}e^{J}}\mathcal{A}\;\frac{dR}{R}\;dt\right|\lesssim (\varepsilon+\frac{1}{J}).\label{int2}
\end{equation}
For $\mathcal{B}$, we recall from \eqref{properrad2} that $|\nabla\varphi(x-y)|+|\nabla \phi(x-y)|\lesssim \frac{1}{R}$. Thus,
$$
\int_{R_{0}}^{R_0e^J}|\nabla[4\varphi(x-y)+\phi(x-y)]|\frac{dR}{R}\lesssim  \int_{R_{0}}^{R_0e^J}\frac{dR}{R^2}\lesssim \frac{1}{R_0}.
$$
Young's inequality and \eqref{unifbond} then give
\begin{equation}
    \left|\frac{1}{T_{0}J}\int_{a}^{a+T_{0}}\int_{R_{0}}^{R_{0}e^{J}}\mathcal{B}\;\frac{dR}{R}\;dt\right|\lesssim \frac{1}{JR_{0}}.\label{int3}
\end{equation}
To deal with $\mathcal{C}$, first we note that  hypothesis \ref{H4}, \eqref{1.17.1}, Lemma \ref{estdiffk}-(i), Young's inequality and \eqref{unifbond} yield
\begin{equation*}
    \begin{split}
        \left|\int_{\R^{5}}\mathrm{Im}\sum_{k=1}^{l}\frac{\alpha_{k}}{\gamma_{k}}f_{k}(\ub(y))\overline{u}_{k}(y)\right|&=\left|\int_{\R^{5}}\mathrm{Im}\sum_{k=1}^{l}\left(\frac{\alpha_{k}}{\gamma_{k}}-\sigma_{k}\right)f_{k}(\ub(y))\overline{u}_{k}(y)\right|\leq \varepsilon_{1}\int_{\R^{5}}\sum_{k=1}^{l}|f_{k}(\ub(y))||u_{k}(y)|\\
        &\lesssim \varepsilon_{1}\|\ub\|^{3}_{\Lb^{3}}\lesssim\varepsilon_{1}.
    \end{split}
\end{equation*}
Also, from Young's inequality and \eqref{unifbond}, we have
$$
\left|\int_{\R^5}\mathcal{T}(\ub(x))dx\right|\lesssim \|\ub\|_{\Hb^1}^2\lesssim 1.
$$
Thus, applying  Cauchy-Schwarz inequality in $\mathcal{C}$,  property $|x-y||\varphi(x-y)|\lesssim 1$ in \eqref{properrad1} and the above estimates we obtain
\begin{equation*}
    \mathcal{C}\lesssim R\varepsilon_{1}. 
\end{equation*}
Hence,
\begin{equation}
    \left|\frac{1}{T_{0}J}\int_{a}^{a+T_{0}}\int_{R_{0}}^{R_{0}e^{J}}\mathcal{C}\;\frac{dR}{R}\;dt\right|\lesssim \frac{R_{0}e^{J}}{J}\varepsilon_{1}.\label{int4}
\end{equation}

Finally, we get a lower bound  for the average of the left-hand side of \eqref{a}, which we will denote by $LHS$. Applying Lemma \ref{lemcoerbal}, there exist $\delta>0$ (and $R$ sufficiently large) such that 
\begin{equation*}
    \begin{split}
        \left|10\int_{\R^{5}}\mathrm{Re}\, F\left(\chi\left(\frac{x-s}{R}\right)\mathbf{u}(x)\right)\;dx\right|\leq  4(1-\delta)\int_{\R^{5}}\chi^{2}\left(\frac{x-s}{R}\right)\mathcal{K}(\ub^{\xi_{0}})\;dx,
    \end{split}
\end{equation*}
uniformly for $t\in \R$. The last inequality combined with   \eqref{identK}  gives
\begin{equation*}
    \begin{split}
       &4\delta K\left(\chi\left(\frac{x-s}{R}\right)\ub^{\xi_{0}}\right)+4\delta\int_{\R^{5}} \chi\left(\frac{x-s}{R}\right)\Delta \left(\chi\left(\frac{x-s}{R}\right)\right)\sum_{k=1}^{l}\gamma_{k}|u_{k}|^{2}\;dx\\
       &\leq 4\int_{\R^{5}}\chi^{2}\left(\frac{x-s}{R}\right)\mathcal{K}(\ub^{\xi_{0}})\;dx-  10\int_{\R^{5}}\mathrm{Re}\, F\left(\chi\left(\frac{x-s}{R}\right)\mathbf{u}(x)\right)\;dx.
    \end{split}
\end{equation*}

By recalling that $\left|\Delta \chi\left(\frac{x-s}{R}\right)\right|\leq \frac{C}{R^{2}}$,  the absolute value of the second term on the left-hand side can be bounded by $\frac{C}{R^{2}}Q(\ub_{0})$. Thus, 
\begin{equation*}
    \begin{split}\label{}
       &4\delta K\left(\chi\left(\frac{x-s}{R}\right)\ub^{\xi_{0}}\right)-\frac{C}{R^{2}}Q(\ub_{0})\\
       &\leq 4\int_{\R^{5}}\chi^{2}\left(\frac{x-s}{R}\right)\mathcal{K}(\ub^{\xi_{0}})\;dx-  10\int_{\R^{5}}\mathrm{Re}\, F\left(\chi\left(\frac{x-s}{R}\right)\mathbf{u}(x)\right)\;dx.
    \end{split}
\end{equation*}

From this and the conservation of the charge we conclude 
\begin{align}
\begin{split}
 &\frac{\delta}{JT_{0}}\int_{a}^{a+T_{0}}\int_{R_{0}}^{R_{0}e^{J}}\frac{1}{R^{5}}\int_{\R^{5}}\left[\M\left(\chi\left(\frac{\cdot-s}{R}\right)\ub\right)K\left(\chi\left(\frac{\cdot-s}{R}\right)\ub^{\xi_{0}}\right)\right]ds\frac{dR}{R}dt\\ 
 &-\frac{C}{JT_{0}}\int_{a}^{a+T_{0}}\int_{R_{0}}^{R_{0}e^{J}}\frac{1}{R^{5}}\int_{\R^{5}}\M\left(\chi\left(\frac{\cdot-s}{R}\right)\ub\right)ds\frac{dR}{R^{3}}dt\\
    &\leq \frac{1}{T_{0}J}\int_{a}^{a+T_{0}}\int_{R_{0}}^{R_{0}e^{J}}(LHS)\;\frac{dR}{R}\;dt.\label{int5}
\end{split}
\end{align}
By noting that that
$$
\int_{\R^{5}}\M\left(\chi\left(\frac{\cdot-s}{R}\right)\ub\right)ds\lesssim R^5\|\chi\|^2_{L^2}Q(\ub_0)
$$
we see that  the second term on the left-hand side of \eqref{int5} can be bounded from below by $\frac{-C}{JR_{0}^{2}}$.  Hence, combining \eqref{a},  \eqref{int1}, \eqref{int2}, \eqref{int3}, \eqref{int4} and \eqref{int5} we deduce
 \begin{align*}
\begin{split}
    &\frac{\delta}{JT_{0}}\int_{a}^{a+T_{0}}\int_{R_{0}}^{R_{0}e^{J}}\frac{1}{R^{5}}\int_{\R^{5}}\left[\M\left(\chi\left(\frac{\cdot-s}{R}\right)\ub\right)K\left(\chi\left(\frac{\cdot-s}{R}\right)\ub^{\xi_{0}}\right)\right]ds\frac{dR}{R}dt\\ &\lesssim \left(\frac{R_{0}e^{J}}{T_{0}J}+\varepsilon+\frac{1}{J}+\frac{1}{JR_{0}}+\frac{R_{0}e^{J}}{J}\varepsilon_{1}+\frac{1}{JR_{0}^{2}}\right).
\end{split}
\end{align*}
By taking $J=\varepsilon^{-2}$, $R_{0}=\varepsilon^{-1}$, $T_{0}=e^{\varepsilon^{-2}}$, and $\varepsilon_{1}=e^{-\varepsilon^{-2}}$, the proof of the proposition is completed. 
\end{proof}

\subsection{Proof of Theorem \ref{thm:critscatt}} This section is devoted to prove our main result. To do so, with the help of Proposition \ref{viri-Morz} we will use the scattering criterion established in Theorem \ref{scatteringcriteio}.

 \begin{proof}[Proof of Theorem \ref{thm:critscatt}]
 	Let $\varepsilon>0$ be small enough and fix $a\in\R$. According to Theorem \ref{scatteringcriteio} it suffices to show that there exists $t_0\in(a,a+T_0)$ such that \eqref{t_0T_0} holds. To do that, first note that \eqref{virmorineq} may be written as
\begin{equation}\label{fl1}
	 \frac{1}{J}\int_{R_{0}}^{R_{0}e^{J}}f(R)\frac{dR}{R}\lesssim \varepsilon,
\end{equation}
 where
 $$
 f(R)=  \frac{\delta}{T_{0}}\int_{a}^{a+T_{0}}\frac{1}{R^{5}}\int_{\R^{5}}\left[\M\left(\chi\left(\frac{\cdot-s}{R}\right)\ub\right)K\left(\chi\left(\frac{\cdot-s}{R}\right)\ub^{\xi_{0}}\right)\right]\;ds\;dt.
 $$
 But from the mean value theorem for definite integrals,
 	 there exists $R\in [R_{0},R_{0}e^{J}]$ so that
 	\begin{equation}\label{fl2}
 		 \frac{1}{J}\int_{R_{0}}^{R_{0}e^{J}}f(R)\frac{dR}{R}=\frac{f(R)}{J}\int_{R_{0}}^{R_{0}e^{J}}\frac{dR}{R}=f(R).
 	\end{equation}
 By combining \eqref{fl1} and \eqref{fl2} we then deduce that
 \begin{equation*}
     \frac{\delta}{T_{0}}\int_{a}^{a+T_{0}}\frac{1}{R^{5}}\int_{\R^{5}}\left[\M\left(\chi\left(\frac{\cdot-s}{R}\right)\ub\right)K\left(\chi\left(\frac{\cdot-s}{R}\right)\ub^{\xi_{0}}\right)\right]\;ds\;dt\lesssim\varepsilon,
 \end{equation*}
for some $R\in [R_{0},R_{0}e^{J}]$.
In particular, from the definition of the functionals $\M$ and $K$  we have for $k=1,\ldots, l$
\begin{equation*}
     \frac{1}{T_{0}}\int_{a}^{a+T_{0}}\frac{1}{R^{5}}\int_{\R^{5}}\left\|\chi\left(\frac{\cdot-s}{R}\right)u_{k}\right\|^{2}_{L^{2}_{x}}\left\|\nabla\left[\chi\left(\frac{\cdot-s}{R}\right)u_{k}^{\xi_{0}}\right]\right\|^{2}_{L^{2}_{x}}\;ds\;dt\lesssim\varepsilon.
 \end{equation*}
  Now, we make the change of variables $s=\frac{R}{4}(z+\theta)$, where $z\in \Z^{5}$ and $\theta\in [0,1]^{5}$ to obtain
 \begin{equation*}
 	\frac{1}{T_{0}}\int_{a}^{a+T_{0}}\sum_{z\in \Z^{5}}\int_{[0,1]^5}\left\|\chi\left(\frac{\cdot-\frac{R}{4}(z+\theta)}{R}\right)u_{k}\right\|^{2}_{L^{2}_{x}}\left\|\nabla\left[\chi\left(\frac{\cdot-\frac{R}{4}(z+\theta)}{R}\right)u_{k}^{\xi_{0}}\right]\right\|^{2}_{L^{2}_{x}}\;d\theta\,dt\lesssim \varepsilon.
 \end{equation*}
The mean value theorem then implies the existence of $\theta_{1}\in [0,1]^{5}$ such that 
 \begin{equation}\label{des1}
     \frac{1}{T_{0}}\int_{a}^{a+T_{0}}\sum_{z\in \Z^{5}}\left\|\chi\left(\frac{\cdot-\frac{R}{4}(z+\theta_1)}{R}\right)u_{k}\right\|^{2}_{L^{2}_{x}}\left\|\nabla\left[\chi\left(\frac{\cdot-\frac{R}{4}(z+\theta_1)}{R}\right)u_{k}^{\xi_{0}}\right]\right\|^{2}_{L^{2}_{x}}\;dt\lesssim \varepsilon.\\
 \end{equation}\\
 
 \noindent {\bf Claim.} There exists $t_{0}\in \left[a+\frac{T_{0}}{2},a+\frac{3T_{0}}{4}\right]$ such that $[t_{0}-\varepsilon^{-\alpha}, t_{0}]\subset [a,a+T_{0}]$, where $\alpha$ is a small constant to be determined later, and
 \begin{equation}\label{clamibeta}
      \int_{t_{0}-\varepsilon^{-\alpha}}^{t_{0}}\sum_{z\in \Z^{5}}\left\|\chi\left(\frac{\cdot-\frac{R}{4}(z+\theta_1)}{R}\right)u_{k}\right\|^{2}_{L^{2}_{x}}\left\|\nabla\left[\chi\left(\frac{\cdot-\frac{R}{4}(z+\theta_1)}{R}\right)u_{k}^{\xi_{0}}\right]\right\|^{2}_{L^{2}_{x}}\;dt\lesssim \varepsilon^{1-\alpha}.
 \end{equation}
 Indeed, by splitting the interval $\in \left[a+\frac{T_{0}}{2},a+\frac{3T_{0}}{4}\right]$ into $T_{0}\varepsilon^{\alpha}$ subintervalos of length $\varepsilon^{-\alpha}$, we can find an interval, say, $[t_{0}-\varepsilon^{-\alpha}, t_{0}]$ such that   \eqref{clamibeta} holds. Otherwise, by the choice of $\alpha$ we would have
 \begin{equation*}
     \int_{a+\frac{T_{0}}{2}}^{a+\frac{3T_{0}}{4}}\sum_{z\in \Z^{5}}\left\|\chi\left(\frac{\cdot-\frac{R}{4}(z+\theta_1)}{R}\right)u_{k}\right\|^{2}_{L^{2}_{x}}\left\|\nabla\left[\chi\left(\frac{\cdot-\frac{R}{4}(z+\theta_1)}{R}\right)u_{k}^{\xi_{0}}\right]\right\|^{2}_{L^{2}_{x}}\;dt\gtrsim T_0\varepsilon^\alpha \varepsilon^{1-\alpha}=T_0\varepsilon,
 \end{equation*}
  which  contradicts \eqref{des1}.  Therefore the claim follows. \\   
 
Recall that our goal is to get the estimate \eqref{t_0T_0}. By interpolation we have
\begin{equation}
	\begin{split}
		\|u_{k}\|_{L^{6}_{t}L^{3}_{x}([t_{0}-\varepsilon^{-\alpha},t_{0}]\times \R^{5})}&\leq  \|u_{k}\|^{\frac{7}{12}}_{L^{\frac{7}{2}}_{t}L^{\frac{7}{2}}_{x}([t_{0}-\varepsilon^{-\alpha},t_{0}]\times \R^{5})} \|u_{k}\|^{\frac{5}{12}}_{L^{\infty}_{t}L_x^{\frac{5}{2}}([t_{0}-\varepsilon^{-\alpha},t_{0}]\times \R^{5})}\\
		&\lesssim \|u_{k}\|^{\frac{7}{12}}_{L^{\frac{7}{2}}_{t}L^{\frac{7}{2}}_{x}([t_{0}-\varepsilon^{-\alpha},t_{0}]\times \R^{5})},
	\end{split}
\end{equation}
where in the last inequality we used the embedding $H^1(\R^5)\hookrightarrow L^{\frac{5}{2}}(\R^5)$ and \eqref{gloestH1}. Another interpolation gives
\begin{equation}\label{rt1}
\|u_{k}\|_{L^{6}_{t}L^{3}_{x}([t_{0}-\varepsilon^{-\alpha},t_{0}]\times \R^{5})} \lesssim \left(\|u_{k}\|^{\frac{5}{13}}_{L^{\frac{5}{2}}_{t}L^{\frac{5}{2}}_{x}([t_{0}-\varepsilon^{-\alpha},t_{0}]\times \R^{5})} \|u_{k}\|^{\frac{8}{13}}_{L^{\frac{14}{3}}_{t}L_x^{\frac{14}{3}}([t_{0}-\varepsilon^{-\alpha},t_{0}]\times \R^{5})}\right)^{\frac{7}{12}}.
\end{equation}

By noting that $\left(\frac{14}{3},\frac{70}{29}\right)$ is an admissible pair and $W^{1,\frac{70}{29}}(\R^5)\hookrightarrow L^{\frac{14}{3}}(\R^5)$, a standard continuity argument yields
\begin{equation}\label{rt2}
	\|u_{k}\|_{L^{\frac{14}{3}}_{t}L_x^{\frac{14}{3}}([t_{0}-\varepsilon^{-\alpha},t_{0}]\times \R^{5})}\lesssim (1+\varepsilon^{-\alpha})^{\frac{3}{14}}.
\end{equation}
Also, using the Gagliardo-Nirenberg inequality $
     \|f\|_{L^{\frac{5}{2}}_{x}(\R^{5})}\lesssim\|f\|^{\frac{1}{2}}_{L^{2}_{x}(\R^{5})}\|\nabla f\|^{\frac{1}{2}}_{L^{2}_{x}(\R^{5})},
$ and the above claim
 we conclude that
 \begin{equation}\label{estpart1}
      \int_{t_{0}-\varepsilon^{-\alpha}}^{t_{0}}\sum_{z\in \Z^{5}}\left\|\chi\left(\frac{\cdot-\frac{R}{4}(z+\theta_1)}{R}\right)u_{k}\right\|^{4}_{L_{x}^{\frac{5}{2}}}\;dt=
       \int_{t_{0}-\varepsilon^{-\alpha}}^{t_{0}}\sum_{z\in \Z^{5}}\left\|\chi\left(\frac{\cdot-\frac{R}{4}(z+\theta_1)}{R}\right)u_{k}^{\xi_0}\right\|^{4}_{L_{x}^{\frac{5}{2}}}\;dt
      \lesssim\varepsilon^{1-\alpha}.
 \end{equation}

On the other hand,  by interpolation, the Sobolev embedding, and the Cauchy-Schwartz inequality we see that
\begin{equation}\label{estpart2}
	\begin{split}
		\sum_{z\in \Z^{5}}&\left\|\chi\left(\frac{\cdot-\frac{R}{4}(z+\theta_1)}{R}\right)u_{k}\right\|^{2}_{L_x^{\frac{5}{2}}}\lesssim \sum_{z\in \Z^{5}}\left\|\chi\left(\frac{\cdot-\frac{R}{4}(z+\theta_1)}{R}\right)u_{k}\right\|_{L_x^{2}} \left\|\chi\left(\frac{\cdot-\frac{R}{4}(z+\theta_1)}{R}\right)u_{k}\right\|_{L_x^{\frac{10}{3}}}\\
		&\lesssim \left(\sum_{z\in \Z^{5}}\left\|\chi\left(\frac{\cdot-\frac{R}{4}(z+\theta_1)}{R}\right)u_{k}\right\|^2_{L_x^{2}}\right)^{\frac{1}{2}}  \left(\sum_{z\in \Z^{5}} \left\|\chi\left(\frac{\cdot-\frac{R}{4}(z+\theta_1)}{R}\right)u_{k}\right\|^2_{L_x^{\frac{10}{3}}}\right)^{\frac{1}{2}}\\
		&\lesssim \|u_{k}\|_{L_x^{2}}\|\nabla u_{k}\|_{L_x^2}\lesssim 1. 
	\end{split}
\end{equation}

Now, using  H\"{o}lder inequality (in the sum)  and then in time,  we deduce in view of \eqref{estpart1} and \eqref{estpart2}
  \begin{equation*}
 	\begin{split}
 		&\|u_{k}\|^{\frac{5}{2}}_{L^{\frac{5}{2}}_{tx}([t_{0}-\varepsilon^{-\alpha}, t_{0}]\times \R^{5})}\\
 		&\lesssim \int_{t_{0}-\varepsilon^{-\alpha}}^{t_{0}}\sum_{z\in \Z^{5}} \left\|\chi\left(\frac{\cdot-\frac{R}{4}(z+\theta_1)}{R}\right)u_{k}\right\|^{\frac{5}{2}}_{L_{x}^{\frac{5}{2}}}\; dt \\
 		&\lesssim \int_{t_{0}-\varepsilon^{-\alpha}}^{t_{0}}\left(\sum_{z\in \Z^{5}}\left\|\chi\left(\frac{\cdot-\frac{R}{4}(z+\theta_1)}{R}\right)u_{k}\right\|^{4}_{L_{x}^{\frac{5}{2}}}\right)^{\frac{1}{4}}\left(\sum_{z\in \Z^{5}}\left\|\chi\left(\frac{\cdot-\frac{R}{4}(z+\theta_1)}{R}\right)u_{k}\right\|^{\frac{3}{4}}_{L_{x}^{\frac{5}{2}}}\right)^{\frac{1}{3}}\;dt\\
 		&\lesssim \left(\int_{t_{0}-\varepsilon^{-\alpha}}^{t_{0}}\sum_{z\in \Z^{5}}\left\|\chi\left(\frac{\cdot-\frac{R}{4}(z+\theta_1)}{R}\right)u_{k}\right\|^{4}_{L_{x}^{\frac{5}{2}}}\;dt\right)^{\frac{1}{4}}\left(\int_{t_{0}-\varepsilon^{-\alpha}}^{t_{0}}\sum_{z\in \Z^{5}}\left\|\chi\left(\frac{\cdot-\frac{R}{4}(z+\theta_1)}{R}\right)u_{k}\right\|^{2}_{L_{x}^{\frac{5}{2}}}\;dt\right)^{\frac{3}{4}}\\
 		&\lesssim \varepsilon^{\frac{1-\alpha}{4}}\varepsilon^{-\frac{3}{4}\alpha}=\varepsilon^{\frac{1}{4}-\alpha}.
 	\end{split}
\end{equation*}
from which we deduce
\begin{equation}\label{rt3}
	\|u_{k}\|_{L^{\frac{5}{2}}_{tx}([t_{0}-\varepsilon^{-\alpha}, t_{0}]\times \R^{5})} \lesssim \varepsilon^{\frac{2}{5}\left(\frac{1}{4}-\alpha\right)}.
\end{equation}
Gathering together \eqref{rt1}, \eqref{rt2} and \eqref{rt3} we obtain
	\begin{equation*}
	 \|u_{k}\|_{L^{6}_{t}L^{3}_{x}([t_{0}-\varepsilon^{-\alpha},t_{0}]\times \R^{5})}\lesssim 
	 \left(\varepsilon^{\frac{2}{5}\left(\frac{1}{4}-\alpha\right)}\right)^{\frac{35}{156}} \left((1+\varepsilon^{-\alpha})^{\frac{3}{14}}\right)^{\frac{56}{156}}.
	\end{equation*}
Finally, by choosing $\alpha>0$ small enough we deduce the existence of some positive constant $\mu$ such that
	\begin{equation*}
	\|u_{k}\|_{L^{6}_{t}L^{3}_{x}([t_{0}-\varepsilon^{-\alpha},t_{0}]\times \R^{5})}\lesssim 
	\varepsilon^{\mu}.
\end{equation*}

 By summing over $k$   we conclude that \eqref{t_0T_0} holds. Hence, an application of Theorem \ref{scatteringcriteio} completes the proof.
\end{proof}

\section*{Acknowledgement}
A.P. is partially supported by CNPq/Brazil grant 303762/2019-5 and FAPESP/Brazil grant 2019/02512-5.



\end{document}